\documentclass[10pt]{article}
\usepackage{amsmath, amsfonts, amsthm, amssymb,verbatim,a4wide}
        \newtheorem{theorem}{Theorem}[section]
\newtheorem{definition}[theorem]{Definition} 
        \newtheorem{lemma}[theorem]{Lemma}
        \newtheorem{proposition}[theorem]{Proposition}
        \newtheorem{corollary}[theorem]{Corollary}
        \newtheorem{remark}[theorem]{Remark}

\numberwithin{equation}{section} 
\usepackage{graphicx}
\usepackage[config, font={rm,md,it}]{subfig}
\usepackage{tikz}
\usetikzlibrary{arrows}
\newcommand \supp {\text{supp}} 
\newcommand \be 	{\begin{equation}}
\newcommand \ee 	{\end{equation}}

\newcommand{\beqn}{\begin{eqnarray}}
\newcommand{\eeqn}{\end{eqnarray}}
\newcommand{\bseqn}{\begin{subeqnarray}}
\newcommand{\eseqn}{\end{subeqnarray}} 
 
\newcommand{\etheo}{\end{theoreme}}
\newcommand{\bprop}{\begin{proposition}}
\newcommand{\eprop}{\end{proposition}}
\newcommand{\blem}{\begin{lemma}}
\newcommand{\elem}{\end{lemma}}
\newcommand{\brema}{\begin{remark}} 

\newcommand{\bdefi}{\begin{definition}}
\newcommand{\edefi}{\end{definition}}  

\newcommand{\C}{{\cal C}}
\newcommand{\calD}{{\mathcal{D}}}

\newcommand{\calB}{{\mathcal{B}}}
\newcommand{\calH}{{\mathcal{H}}}
\newcommand{\calW}{{\mathcal{W}}}
\newcommand{\calO}{{\mathcal{O}}}
\newcommand{\wk}{{w^n_{K}}}
\newcommand{\wke}{{w^n_{K_e}}}
\newcommand{\wkve}{{w^n_{K,e}}}
\newcommand{\wkeve}{{w^n_{K_e,e}}}
\newcommand{\wwk}{{w^{n+1}_{K}}}
\newcommand{\wwkve}{{w^{n+1,-}_{K,e}}}
\newcommand{\wwkeve}{{w^{n+1,-}_{K_e,e}}}

\newcommand{\ve}{{v_e}}
\newcommand{\vk}{{v_K}} 
\newcommand{\uk}{{u^n_K}}
\newcommand{\uke}{{u^n_{K_e}}}

\newcommand{\uuk}{{u^{n+1}_K}} 
\newcommand{\Th}{{\cal T}_h}
\newcommand{\Thstar}{{\Th}^\star}
\newcommand{\volk}{{\vert K\vert}}
\newcommand{\volke}{{\vert K_e\vert}}

\newcommand{\vole}{{\vert e\vert}}
\newcommand{\pk}{{p_K}}
\newcommand{\sumnn}{
\displaystyle{\sum_{n \ge 0}~}}

\newcommand{\sumek}{
\displaystyle{\sum_{e\in \partial K}}}
\newcommand{\sumkT}{\displaystyle{\sum_{K\in{\cal T}_h~}}}

\newcommand{\nuek}{{\nu_{K,e}}}
\newcommand{\nueke}{{\nu_{K_e,e}}}

\newcommand{\ks}{{K^\star}}
\newcommand{\es}{{e^\star}}
\newcommand{\nuesk}{{\nu_{\ks(e),\es}}}
\newcommand{\sumesdk}{{\sum_{\es\in \ks(e)\cap K}}}
\newcommand{\voles}{{\vert\es\vert}} 
\newcommand{\gek}{{g_{e,K}}}
\newcommand{\Gek}{{G_{e,K}}}

\renewcommand{\u}{{u}}
\renewcommand{\v}{{v}}
\newcommand{\uh}{{u_h}}
\newcommand{\vh}{{v_h}}
\newcommand{\wh}{{w_h}}
\newcommand{\w}{{w}}
\newcommand{\cz}{{\cal C}_0}
\newcommand{\ci}{{\cal C}_i}

\newcommand{\ffi}{{f}_i}

\newcommand{\lil}{{\ell}_i^l}

\renewcommand{\ll}{{\ell}}
\newcommand{\flux}{{f}}
\newcommand{\Q}{{\cal Q}}
\newcommand{\U}{{\cal U}}
\newcommand{\F}{{\cal F}}
\renewcommand{\L}{{\cal L}}
\newcommand{\Lil}{{\cal L}_i^l}
\newcommand{\dt}{{\partial_t}}
\newcommand{\dxi}{{\partial_{x_i}}} 
\newcommand{\ake}{{\alpha_{K,e}}}
\newcommand{\akee}{{\alpha_{K_e,e}}}

\newcommand{\cpl}{\theta}
\newcommand{\cplL}{\theta_-}
\newcommand{\cplR}{\theta_+}

\newcommand{\lpc}{\gamma}
\newcommand{\lpcL}{\gamma_-}
\newcommand{\lpcR}{\gamma_+}
\newcommand{\lpcLR}{\gamma_\pm}
\def \dessin[#1,#2,#3,#4] {
{
\centerline{
  \epsfysize=#2
  \leavevmode\epsfbox{#1} }}
\centerline{#3  #4}
\smallskip
}
\def\longrightharpoonup{\relbar\joinrel\rightharpoonup}
\def\R{{\mathord{I\!\! R}}}

\begin{document}

\title{Coupling techniques for nonlinear hyperbolic equations. IV. 
Multi--component coupling and multidimensional well--balanced schemes}
\author{
Benjamin Boutin$^1$, Fr\'ed\'eric Coquel$^2$, and Philippe G. LeFloch$^3$}

\date{\today}
\maketitle

\footnotetext[1]{Institut de Recherche Math\'ematiques de Rennes, 
Universit\'e de Rennes 1, Campus de Beaulieu, 35042 Rennes,
France.}

\footnotetext[2]{Centre de Math\'ematiques Appliqu\'ees \& Centre National de la Recherche Scientifique, Ecole Polytechnique, 91128 Palaiseau, France.}

\footnotetext[3]
{Laboratoire Jacques-Louis Lions \& Centre National de la Recherche Scientifique,
Universit\'e Pierre et Marie Curie, 4 Place Jussieu,  75252 Paris, France.
\\
Email: {\sl benjamin.boutin@univ-rennes1.fr, coquel@cmap.polytechnique.fr, contact@philippelefloch.org.}  
}

\begin{abstract}
This series of papers is devoted to the formulation and the approximation of coupling problems for nonlinear hyperbolic
equations. The coupling across an interface in the physical space is formulated in term of an augmented system of partial differential equations. In an earlier work, this strategy allowed us to develop a regularization method based on a thick interface model in one space variable. In the present
paper, we significantly extend this framework and, in addition, encompass equations in several space variables. This new formulation includes the coupling of several distinct conservation laws and allows for a possible covering in space.  Our main contributions are, on one hand, 
 the design and analysis of a well--balanced finite volume method on general triangulations
and, on the other hand, a proof of convergence of this method toward entropy solutions, extending Coquel, Cockburn, and LeFloch's theory (restricted to a single conservation law without coupling). The core of our analysis is, first, the derivation of entropy inequalities as well as a discrete entropy dissipation estimate and, second, a proof of convergence toward the entropy solution of the coupling problem.  
\end{abstract}



\section{Introduction}
\label{INT} 

\subsubsection*{Objective of this paper}

This is a continuation of a series of papers \cite{BoutinCoquelLeFloch09a,BoutinCoquelLeFloch09b,BoutinCoquelLeFloch09c} 
devoted to coupling techniques for nonlinear hyperbolic
equations. In the present paper, we deal with the coupling of {\sl multi-dimensional hyperbolic equations,}
based on an arbitrary partition of the
physical domain. The main motivation stems from the study of complex systems resulting from the
combination of elementary components modeled by different equations. Indeed, each component may
be subject to physical phenomena involving fairly different time and space scales. Tackling this multiscale problem
with sufficient accuracy and efficiency requires to consider distinct physical models for the description of
each component, so as to end up with a suitable description of the whole physical system. For instance, large--scale
power plants provide a typical example of interest~\cite{neptune}.
Describing the evolution in time requires the exchange of transient informations
at each physical boundary separating two distinct hyperbolic models. These transient informations or boundary conditions 
are referred hereafter to as \emph{coupling conditions}.

This problem seems to be rather new in the applied mathematical community. Its analysis was initiated by Godlewski and Raviart \cite{GodlewskiRaviart04} for scalar equations in one space variable. Therein, the coupling problem is formulated in terms of two initial
boundary value problems (IBVP) supplemented with coupling boundary conditions at a given (infinitely
thin) interface. These boundary conditions are stated in such a way that in ``most cases'' they ensure the
continuity of the main unknown, at least, roughly speaking, as long
as no wave from the left-- and right--hand problems interact at the interface. If this condition does not hold, 
one says that the interface is resonant. In Ambroso \emph{et al.} \cite{GDT06,GDT07,GDT08b},
quite general continuity conditions based on a nonlinear transformation of the unknown were investigated. Following earlier investigations by LeFloch and collaborators~\cite{DLM,JL,LeFloch-93,LeFloch-book,LeFloch-review,LL}  on undercompressive shocks and interfaces, nonconservative hyperbolic systems, and boundary value problems, we stress that {\sl additional information} coming from physical modeling is necessary in order to single out the relevant continuity conditions (or transmission condition) at the interfaces.
Various conditions were introduced and studied in a variety of physical frameworks, ranging from gas
dynamics \cite{GDT06} to multiphase flows \cite{GDT08c,GDT07}. 


\subsubsection*{Thin interface versus thick interface}

We briefly mention some transmission conditions of interest when the coupling invoves two Euler systems with
distinct pressure laws. Typically, one imposes the continuity of the density $\rho$, velocity component $u$, and
pressure $p$, or else the continuity of the convervative variables $(\rho,\rho u,\rho E)$ (where $E$ denotes the total energy). 
These conditions determine the class of constant solutions in the time and the space variables, and have 
either constant density, velocity, and pressure, or else constant density, momentum, and total energy.
In both cases, the proposed coupling conditions are {\sl nonconservative}, since 
the total mass of density, momentum, and total energy do vary with time. A fully conservative coupling may
turn relevant in some applications, as was addressed in \cite{GDT08b} (following \cite{GodlewskiRaviart04}) via suitable a relaxation method. 

The resonance phenomena, likely to take place around thin interfaces, brings a main difficulty
in the mathematical analysis of coupled initial boundary value problems. Solutions can be
shown to exist under general conditions but resonance generally comes at the expense of uniqueness. We
refer the reader to \cite{BoutinChalonsRaviart09} for a discussion of scalar equation and to \cite{GDT06} for a distinct behavior exhibited for characteristic but non-resonant interfaces.
A selection criterion for discontinuous solutions, therefore, is required. Recall
that, for the fully conservative coupling, several distinct entropy criteria have been proposed, each
selecting a distinct weak solution in agreement with the physical context. (See \cite{BurgerKarlsen08} for a review and \cite{SeguinVovelle03,HelluySeguin06,BachmannVovelle06}).

To deal with general transmission conditions, a macroscopic selection principle analogous to the entropy
inequalities is not available and one needs a detailed description of microscopic mechanisms coming
with suitable regularizing procedures. In \cite{BoutinCoquelGodlewski06, BoutinCoquelLeFloch09a,BoutinCoquelLeFloch09b}, we introduced an
alternative modeling for the coupling problem associated with two hyperbolic equations in one space variable. This alternative method relies on the introduction of an augmented PDE (partial differential equations) formulation that avoids the detailed description of the interfaces. The proposed formalism is based on an additional unknown, the \emph{color function} which takes values in the range $[0,1]$. Extreme values $0$ and $1$ are devoted to restore the left-- and right-- problems to be coupled, while intermediate values may serve to model a smooth transition from one problem to the other.


\subsubsection*{Outline of this paper}

The interest in this augmented formulation comes from its very capability to support various
regularization mechanisms. Viscous perturbations were introduced by the authors \cite{BoutinCoquelLeFloch09a,BoutinCoquelLeFloch09b} for scalar problems and, specifically, a self--similar approach was developed, which allows for the study of the existence and
uniqueness of solutions to the coupled Riemann problem in the limit of vanishing viscosity. The analysis has been carried out for a general class of systems \cite{BoutinCoquelLeFloch09a} and led to an existence theory 
under fairly general assumptions. In \cite{BoutinCoquelLeFloch09b}, the analysis of the internal
structure of resonant interfaces was performed and led us to a characterization of the set of admissible Riemann solutions. Despite of the viscous mechanisms a failure of uniqueness may be observed for resonant
infinitely thin interfaces.

Riemann solutions may be indeed understood as
describing the long time asymptotic of the solutions of the Cauchy problem. Failure of uniqueness for thin
interfaces just reflects the property that distinct regularizations of thin interfaces may give rise to
different solutions and thus with a distinct long time behavior. This observation has motivated a second
regularization procedure based on \emph{thick interfaces}.

Thick interfaces within the augmented PDE framework are based on a 
regularization of the discontinuous color function, considered in the thin regime. This approach
has been introduced by the authors first within the framework of two coupled conservation laws in one space variable
\cite{BoutinCoquelLeFloch09c}. Existence and uniqueness for the coupled Cauchy problem was proven for 
general initial data with bounded sup--norm. One of the main ingredients of proof was the
design of a well--balanced finite volume method. The well-balanced property means that the exact
constant solutions selected by a given transmission condition are exactly preserved at the discrete level, whatever choice is made for the regularized color function. This consistency property is of central importance.

In the present paper, we introduce a framework which covers coupling problems in {\sl several} space variables and with {\sl distinct} hyperbolic equations, allowing for possible {\sl covering in space}. An outline of this paper is as follows. 

\begin{itemize}

\item  In Section~2, we show how to extend the two existing coupling
frameworks in one space variable to the coupling of two distinct hyperbolic equations in several space variables. We then show how to extend the augmented PDE formalism to encompass the case of several
hyperbolic equations with possible covering. In our approach, a vector--valued color map is introduced so that each component is
 associated with one of the equations and takes values in the interval $[0,1]$. The specific definition of the regularized color function provides us with 	a transition from an equation to another (possibly more than one). 

\item We check the existence and uniqueness of entropy solutions to the coupled Cauchy problem (with
initial data in $L^\infty$) under fairly general assumptions on the transmission conditions and the
equations under consideration.

\item Next, in Section~3, we design a robust and flexible finite volume framework based on
general triangulations. Importantly, by construction, the proposed method is well--balanced and our strategy for achieving
the well--balanced property is an extension of the subcell reconstruction approach (analyzed by 
Bouchut in a different context \cite{Bouchut04}). In particular, we introduce two distinct meshes: the first one, the primal mesh,
 describes the main coupled unknown. The second mesh, referred to as the dual one, is built from the primal mesh and carries 
the approximation of the color function. A comprehensive derivation of this dual mesh is also proposed.

\item In Sections~4 and 5, 
we then derive a sup--norm estimate, and observe that a uniform estimate on the total variation seems to be out of reach, 
due to the subcell reconstruction procedure. Consequently, we propose to use DiPerna's framework based on entropy
measure--valued solutions and, by deriving suitable entropy inequalities and entropy dissipation bounds, 
we establish the strong convergence of the proposed method. 

\item Finally, in Section~6, numerical experiments are presented which concern problems with covering in space and, therefore, highlight the interest of the new coupling strategy.

\end{itemize}


\section{A framework for multi--dimensional coupling} 
\label{formulation}

\subsection{Coupling of two systems} 

\subsubsection*{Pasting together two initial boundary value problems}

In this section, we introduce the coupling problem associated with two
hyperbolic equations coupled at a given interface. At this stage, it suffices to think of an
hyperplane, say $\{x_1=0\}$. We will extend two distinct coupling strategies that have been
developed in a single space variable.
The first procedure consists in modeling the coupling problem as two initial boundary value problems
(IBVP) with time dependent boundary conditions prescribing the evolution of traces of the coupled
solutions on both sides of the hyperplane $\{x_1=0\}$.
In contrast, the second strategy introduced in
\cite{BoutinCoquelLeFloch09a,BoutinCoquelLeFloch09b,BoutinCoquelLeFloch09c} is based on augmented PDE
systems, and handles the coupling problem as an initial data problem written over the entire space $\R^d$.
This new framework brings mathematical and numerical advantages, pointed out at the end of this
section.

Consider an hyperplane of $\R^d$ with unit normal vector $\nu\in\R^d$, we denote $\mathcal{H}
= \{x\in\R^d / x.\nu = 0\}$, partitioning $\R^d$ into two half-domains $\calD_- = \{x\in\R^d / x.\nu <
0\}$ and $\calD_+ = \{x\in\R^d / x.\nu > 0\}$. In each open subdomain, a distinct conservation
law is prescribed:
\be
\label{multid-cl}
 \dt \w + \sum_{i=1}^d \dxi a_i^{\pm}(\w) = 0,\qquad \w(t,x)\in\R,\quad t>0,\quad x\in \calD_\pm,
\ee
where the flux-functions $A^\pm:\R\to\R^d$, with components $(a_i^\pm)_{i=1,\ldots,d}$, are assumed to be
twice differentiable for definiteness. An initial data $\w(0,x)=\w_0(x)$ supplements this formulation, but
obviously, some extra-condition, the \emph{coupling condition}, must be prescribed at the
interface $\calH$. For simplicity, we restrict ourselves in this introductory section to piecewise smooth
solutions $w$ with bounded left and right traces at the interface $\calH$:
$$
  \w(t,y\pm) = \lim_{z\to 0+} \w(t,y\pm z \nu),\quad y\in\calH.
$$
Then, it sounds natural that the coupling condition we seek should relate these traces
\be
 \label{star1}
 \mathfrak{C}(\w(t,y-),\w(t,y+))=0,\quad t>0,\ y\in \calH,
\ee
for some nonlinear mapping $\mathfrak{C}$ to be specified.
The implicit function theorem is assumed to apply so as to recast \eqref{star1} in the more tractable
form
\be
\label{star2}
 \w(t,y-) = \mathfrak{c}(w(t,y+)),\quad t>0,\ y\in \calH,
\ee
for some function $\mathfrak{c}$ mapping $\R$ onto $\R$.
Assuming from now on $\mathfrak{c}$ to be strictly monotone, we re-express the
above coupling condition in terms of two nonlinear monotone functions $\cplL$ and $\cplR$ with
$\mathfrak{c}=\cplL^{-1}\circ\cplR$:
\be
\label{multid-coupling}
 \cplL(\w(t,y-)) = \cplR(\w(t,y+)),\quad t>0,\ y\in \calH.
\ee
Here and without loss of generality, $\cplL$ and $\cplR$ are assumed to be strictly increasing and to map
$\R$ onto $\R$, and their inverse functions are denoted by $\lpcL$ and $\lpcR$.
On the basis of this pair of functions, we introduce the following useful change of unknown:
\be
 \label{star4}
 u(t,x) =
\begin{cases}
 \cpl_-(\w(t,x)), &t>0,\ x\in\calD_-,\\
 \cpl_+(\w(t,x)), &t>0,\ x\in\calD_+,\\
\end{cases}
\ee
so that the coupling condition \eqref{multid-coupling} resumes to:
\be
 \label{star5}
 u(t,y-)=u(t,y+),\quad y\in\calH.
\ee
Observe that in the new unknown, \eqref{star5} juste reads as a continuity condition for $u$.  

It is worth underlining that \eqref{star5} defines the constant solutions
of the coupling problem \eqref{multid-cl}-\eqref{multid-coupling}, i.e.~time independent functions $w=w(x)$
which solve \eqref{multid-cl} and \eqref{star5}. Such functions clearly obey 
\be
 \label{star6}
 u(w(x))=u^\star,\quad x\in\R^d\setminus\calH,
\ee
for some real $u^\star\in\R$.
This observation actually just opens a path toward the mathematical study of perturbed
solutions of the trivial solution \eqref{star6}. We refer the reader to the work
\cite{BoutinCoquelLeFloch09a} devoted to the existence of self--similar coupled solutions for systems.

Observe that the coupling condition \eqref{multid-coupling} plays the role of a pair of transient boundary
conditions for the interface $\calH$. In other words, the coupling framework we address merely takes the
form of two nonlinear hyperbolic IBVPs linked via the transient boundary condition
\eqref{multid-coupling}. It becomes clear that the coupling condition
\eqref{multid-coupling} is actually expressed in a strong sense, since it is formulated without reference
to the signature of the wave speeds at the interface $\calH$. It is nevertheless well-known that the sign
of the wave velocities at a boundary directly affects the boundary condition to be prescribed. Hence, the
coupling condition \eqref{multid-coupling} or its equivalent form \eqref{star5} must be stated in a  weak sense. 

We follow the approach for coupled problem in {\sl one space variables,} originally developed in Godlevski, Raviart, and collaborators 
(cf.~\cite{GodlewskiLeThanhRaviart05,GodlewskiRaviart04} and \cite{GDT06,GDT07, ChalonsRaviartSeguin08}). 
In these papers, a weak form of the coupling condition
\eqref{multid-coupling} is formulated in terms of an {\sl admissible boundary set,} proposed by 
Dubois and LeFloch \cite{DuboisLeFloch88} and based on the notion of Riemann solutions.
Such a notion here readily extends since the coupling condition expressed in \eqref{multid-coupling} just
links the traces of the coupled solution $w$ in the normal direction $\nu$ and thus essentially concerns
the quasi-one dimensional form of \eqref{multid-coupling} written for plane wave solution in the
$\nu$-direction. Thus it turns natural to consider the coupled problem in one space variable (up to some shift in the space variable $z$)
\be
\label{multid-1Dform}
 \dt \w + \partial_z A_\nu^\pm(\w) = 0, \quad t>0,\ \pm z>0,
\ee
where we have set
$A_\nu^\pm(\w) = A^\pm(\w) \cdot \nu$. In order to state the weak form of the boundary condition $\cplL(\w(t,y-)) = \cplR(\w(t,y+))$,
$y\in\calH$, we first recall the Dubois-LeFloch framework for say the right IBVP:
\begin{eqnarray}
  &&\dt \w + \partial_z A_\nu^+(\w) = 0, \quad t>0,\ z>0,\label{star7}\\
  &&w(t,0+)=b,\quad t>0,\label{star8}
\end{eqnarray}
for some prescribed real  $b$. Following Dubois and LeFloch, a  weak formulation of
\eqref{star8} is stated in terms of Riemann solutions associated with \eqref{star7}, that is, 
$\calW(\cdot;w_L,w_R)$ (for left-- and right--hand states $w_L$, $w_R$):  
\be
\label{star9}
 w(t,0+)\in\calO_\nu^+ (b)= \big\{\calW(0+;b,w), w\in\R\big\}.
\ee
Observe that the analogous of \eqref{star9} for the left IBVP built from $A_\nu^-$ would read
$$
  w(t,0-)\in\calO_\nu^-(b)= \big\{\calW(0-;w,b), w\in\R\big\}.
$$
These considerations naturally yield us to the following coupled boundary conditions
\eqref{multid-coupling} at any point $y\in\calH$ and for $t>0$: 
\be
\label{weakcpl}
 \begin{aligned}
  \w(t,y+) \in \mathcal{O}_\nu^-(\cplR^{-1}\circ \cplL(\w(t,y-))),\\
  \w(t,y-) \in \mathcal{O}_\nu^+(\cplL^{-1}\circ \cplR(\w(t,y+))).
 \end{aligned} 
\ee

This simple problem, based of two coupled equations at a given hyperplane, can be easily extended to more general
interfaces resulting from a partition of $\R^d$ into two non--overlapping open sets $\calD_+$ and $\calD_-$
such that $\overline{\calD}_-\cup\overline{\calD}_+=\R^d$, separated by a smooth boundary $\partial
D=\overline{\calD}_-\cap\overline{\calD}_+$. Smoothness allows to define without ambiguity an unit normal
vector $\nu(y)$ for all $y\in\partial D$ so that left and right traces at $\partial D$ for piecewise
smooth solutions of the coupled problem \eqref{multid-coupling} may be defined as follows:
$$
 \w(t,y\pm) = \lim_{z\to 0+} \w(t,y\pm z \nu(y)),\quad y\in\partial\calD.
$$
The expected coupling condition just takes the weak form \eqref{weakcpl}.


\subsubsection*{Coupling technique based on an augmented PDE's system}

As already emphasized, an alternative coupling framework has been introduced by the authors in
\cite{BoutinCoquelLeFloch09a}. Instead of dealing with two IBVPs coupled at a given interface via boundary conditions, our
 new approach treats the coupling problem as a single initial
value problem, over the entire space $\R^d$ via an augmented PDE formulation. This strategy 
was introduced by the authors \cite{BoutinCoquelLeFloch09a,BoutinCoquelLeFloch09b,BoutinCoquelLeFloch09c} for
problems  in one space variable. In order to encompass problems in several space variables, we
perform hereafter a comprehensive derivation.

The derivation starts from the characteristic functions of the two open sets $\calD_-$ and $\calD_+$, we
denote by
$$
 v_-=\chi_{\calD_-},\quad v_+=\chi_{\calD_+}.
$$
It heavily makes use of the change of unknown $u$ introduced in \eqref{star4}, we rephrase as:
$$
 u(t,x)=
\begin{cases}
 \cpl_-(\w(t,x)), &\textrm{if } v_-(x)=1,\\
 \cpl_+(\w(t,x)), &\textrm{if } v_-(x)=0,\textrm{ i.e. if } v_+(x)=1,\\
\end{cases}
\quad t>0,\ x\notin\partial\calD.
$$
Equipped with these notation, we recast the two distinct hyperbolic equations in $\calD_\pm$ in terms of $u$:
$$
 \lpcLR'(u)\dt u + \sum_{i=1}^d \lpcLR'(u) {a_i^\pm}'(\lpcLR(u)) \dxi u = 0,\quad t>0,\ x\in\calD_\pm,
$$
restricting ourselves to smooth solutions in a first stage. Recall that $\lpcR$ (respectively $\lpcL$)
denotes the inverse function of $\cplR$ (resp. $\cplL$).
We further proceed by rewritting the above two equations in term of a single equation in $x\in R^d\setminus\partial D$:
$$
\big(v_-\lpcL'(u)+v_+\lpcR'(u)\big)\dt u
+ \sum_{i=1}^d
\Big(v_-\lpcL'(u){a_i^-}'\!(\lpcL(u))+v_+\lpcR'(u){a_i^+}'\!(\lpcR(u))\Big)\dxi u = 0,
$$

At this stage, it must be noticed that the two characteristic functions $v_-$ and $v_+$ in the above
equation may be replaced by a single function say $v$, by setting for instance $v_-(x) = 1-v(x)$ and
$v_+(x)=v(x)$ for $x\in\R^d\setminus\partial D$ with $v=\chi_{\calD_+}$. In the following, such a function
$v$ will be refered to as a \emph{color function}. For the moment $v$ is nothing but a step function
taking values in $\{0,1\}$ but it is important to conceive $v$ as a function taking values in the interval
$[0,1]$ so that the value 0 restores the equation set in $\calD_-$ while the value 1 restores the equation
set in $\calD_+$. Intermediate values of $v$ then may be thought as modeling a smooth shift from one
problem to the other. Keeping this in mind we now recast the equations above in the form of an augmented PDE
system with unknown $u$ and $v$, for $t>0$ and $x\in R^d\setminus\partial D$:
\be
 \label{star11}
\begin{aligned}
 &\big((1-v)\lpcL'(u)+v\lpcR'(u)\big)\dt u 
+ \Big((1-v)\lpcL'(u)\nabla A^-\!(\lpcL(u))+v\lpcR'(u)\nabla A^+\!(\lpcR(u))\Big)\cdot\nabla_x u =
0,\\
 &\dt v =0.
\end{aligned}
\ee
We stress that the $1$--dimensional form of these equations written for plane wave solutions in the direction
$\nu$ reads ($t>0$, $x\in R^d\setminus\partial D$, or $\pm z >0$):
\be
\label{planeequation}
  \begin{aligned}
 &\big((1-v)\lpcL'(u)+v\lpcR'(u)\big)\dt u + \Big((1-v)\lpcL'(u)\nabla A^-\!(\lpcL(u))\cdot\nu+v\lpcR'(u)\nabla
A^+\!(\lpcR(u))\cdot\nu\Big)\partial_z u = 0,\\
 &\dt v =0.
\end{aligned}
\ee
This system is easily seen to be hyperbolic if (and only if) the following quantity is not zero
\be
 \label{star12}
 (1-v)\lpcL'(u)\nabla A^-(\lpcL(u))\cdot\nu+v\lpcR'(u)\nabla A^+(\lpcR(u))\cdot\nu \neq 0.
\ee
For such states, the standing wave associated with the additional unknown $v$ can be seen to admit $u$ as
a Riemann invariant. In other words, as long as the non--degeneracy condition \eqref{star12} is valid, $u$
stays continuous at the jumps of the color function $v$, namely across the coupling boundary $\partial D$
at which the value of $v$ shifts from 0 to 1. In other words and whenever \eqref{star12} is valid, the
coupling condition \eqref{star5} is satisfied in the strong sense across the standing wave
\be
\label{star13}
 u(t,y-) = u(t,y+), \quad y\in\partial D.
\ee
Violation of the condition \eqref{star12} at a point of jump for $v$, namely at the interface
$\partial\calD$, expresses that waves from the left and right propagate with opposite sign at the
interface; the first order system \eqref{planeequation} is then only weakly hyperbolic. This is the
 resonance phenomena for which we refer the reader to, for instance, Goatin and LeFloch
\cite{GoatinLeFloch04} and the references cited therein. As far as the coupling issue is concerned, the continuity condition \eqref{star13}
is no longer satisfied and the weak form \eqref{weakcpl} of the coupling condition must be addressed.
Turning considering the augmented formulation \eqref{planeequation}, resonance phenomena has been studied
in depth in \cite{BoutinCoquelLeFloch09a} in the scalar case thanks to a self-similar viscous
perturbation. The Riemann solutions for \eqref{planeequation} defined in the limit of vanishing
viscosity satisfy \eqref{weakcpl} when resonance takes place. To sum up, weak solutions of
the augmented equations \eqref{planeequation} and thus their multi--dimensional form \eqref{star11}
naturally encode the weak form of the coupling condition.

We now generalize the rather special form of the augmented equation and adopt the
general framework introduced by the authors in \cite{BoutinCoquelLeFloch09a} (which also applies to systems in one
space variable). We thus introduce coupling functions $\cz:\R\times[0,1]\to \R$ and $\ci:\R\times[0,1]\to
\R$ with $i\in\{1,\ldots,d\}$ satisfying the following consistency properties:
\be
\label{consistency}
\begin{array}{ll}
\lim_{\v\to 0}\cz(\u,\v) = \lpcL(\u),& \lim_{\v\to 1}\cz(\u,\v) = \lpcR(\u),\\
\lim_{\v\to 0}\ci(\u,\v) = a_i^-(\lpcL(\u)), &\lim_{\v\to 1}\ci(\u,\v) =  a_i^+(\lpcR(\u)),
\end{array}
\ee
so as to consider in place of \eqref{star11} the general augmented equations:
\be
\label{star14}
\begin{aligned}
\displaystyle
\partial_\u \cz(\u,\v) \dt \u + \sum_{i=1}^d \partial_\u \ci(\u,\v) \dxi \u & = 0, \\
\dt \v & = 0,
\end{aligned}
\qquad t>0,\ x\in\R^d,
\ee
which equivalently recasts as:
$$
\begin{aligned}
\displaystyle
\dt \cz(\u,\v) + \sum_{i=1}^d \dxi \ci(\u,\v) - \sum_{i=1}^d \partial_v \ci(\u,\v) \dxi \v & = 0,\\
\dt \v & = 0.
\end{aligned}
\qquad t>0,\ x\in\R^d.
$$
In the following, the coupling functions $\cz$ and $(\ci)_{1\leq i \leq d}$ are smooth and
$$
 \cz,(\ci)_{1\leq i \leq d} \in \mathcal{C}^2(\R\times[0,1]),
$$
and $\cz$, in addition, obeys 
$\partial_\u \cz(\u,\v) > 0,\quad \u\in\R,\ \v\in[0,1]$, which is
 a non--degeneracy condition for the time arrow in \eqref{star14}.

The resonance phenomenon is the main difficulty in the coupling problematic and has made the matter of
previous works especially in the one--dimensional case
\cite{BoutinCoquelLeFloch09a,BoutinCoquelLeFloch09b,BoutinCoquelLeFloch09c}.
In this one-dimensional setting, the analysis proves that if the resonance occurs for \eqref{star14} the
self--similar weak solutions obtained via self--similar regularization satisfy the
coupling relation \eqref{star13}. Nevertheless in the general case where resonance may appear, uniqueness
then generaly fails for the initial value problem.

The central interest of the augmented formulation \eqref{star14} over more classical coupling approaches
built from a collection of IBVPs stems from the fact it can be supplemented with a variety of regularizing
mechanisms at the coupling interfaces.
These regularization mechanisms are intended to handle the resonance phenomena which is likely to take
place at the interfaces.
A first regularization procedure relies on introduction of suitable viscous mechanisms. Such mechanisms
yield a non trivial internal structure to resonant interfaces which proves to be useful in the selection
of discontinuous solutions.
It turns that discontinuous solutions may not be unique for thin interfaces. The augmented formulation
\eqref{star14} actually allows for another regularization mechanism based on thick interfaces. 
The color function which is naturally discontinuous (for the description of thin interfaces) is 
regularized in the thick regime. Such a regularization technique has been analyzed in one space
variable, and existence and uniqueness of a solution for the Cauchy problem has been established. In the next
section, we show how to extend this regularization procedure to several space variables.

\begin{remark}
\label{example2}
An example of coupling functions satisfying the above conditions is 
$$
\begin{aligned}
 \cz(\u,\v)&=(1-\v)\lpcL(\u)+\v\lpcR(\u),\\
 \ci(\u,\v)&=(1-\v)a_i^-(\lpcL(\u))+\v a_i^+(\lpcR(\u)),\ 1\leq i \leq d.
\end{aligned}
$$
\end{remark}


\subsection{A framework for multi--component coupling problems}

\subsubsection*{Multi-component coupling of initial boundary value problems}

We are in a position to present the general coupling framework we intend to analyze in this paper. The
proposed extension treats the coupling of $(L+1)$, $L\geq 1$, distinct conservation laws in several space
dimensions, with possible covering. The coupling modeling via augmented PDEs relies on a partition of the
space $\R^d$ in a finite number of non--overlapping, non--empty and open sets $(\calD^l)_{0\leq l \leq L}$:
\be
 \bigcup_{l=0}^L \overline{\calD^l} = \R^d.
\ee
The set of boundaries $\calB$ are given by
\be
 \calB = \bigcup_{k\neq l} \overline{\calD_k}\cap \overline{\calD_l}.
\ee
An interface $\calH_{kl}$ is by definition the part of the boundary of $\calD_k$ which is only shared with
$\calD_l$ (see also Fig.~\ref{intro-multiD} for an example with $N=2$ and $L=3$):
\be
\calH_{kl} = (\overline{\calD_k}\cap \overline{\calD_l}) \setminus \bigcup_{i\neq
k,l}\overline{\calD_i}.
\ee
These interfaces $\calH_{kl}$ are supposed to be smooth enough so that they admit an unit normal vector $\nu_{kl}(y)$, which is well--defined except at some ``exceptional'' points (like corners, etc.). 
 We suppose the set of boundaries $\calB$ to be of $d$-dimensional Lebesgue measure zero, and,
more precisely, the remaining set $\calB \setminus (\cup_{k\neq l}\calH_{kl})$ has only components of
Hausdorff dimension less than or equal to $(d-2)$ (see for example the four points underlined in
Figure~\ref{intro-multiD}).

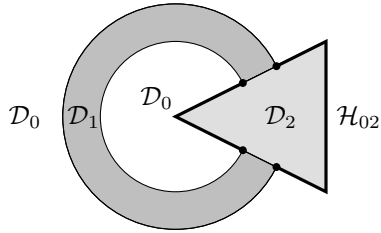
\begin{figure}[!ht]
\centering
\begin{tikzpicture}[scale=1.0,
 dot/.style={shape=circle,fill=black,minimum size=3pt,inner sep=0pt,outer sep=2pt}]
\draw[fill=lightgray] (0,0) ellipse (1.5 and 1.5);
\draw (0,0) ellipse (1.5 and 1.5);
\draw[fill=white] (0,0) ellipse (1 and 1);
\draw[fill=white] (0,0) -- (2,1) -- (2,-1) -- cycle;
\draw[fill=lightgray!50!white] (0,0) -- (2,1) -- (2,-1) -- cycle;
\draw (0,0) -- (2,1) -- (2,-1) -- cycle;
\node at (-2,0) {$\calD_0$};
\node at (-0.25,0.25) {$\calD_0$};
\node at (-1.22,0) {$\calD_1$};
\node at (1.4,0) {$\calD_2$};
\draw[very thick] ({2/sqrt(5)},-{1/sqrt(5)}) node[dot]{} -- (0,0) -- ({2/sqrt(5)},{1/sqrt(5)})
node[dot]{};
\draw[very thick] ({1.5*2/sqrt(5)},-{1.5*1/sqrt(5)}) node[dot]{} -- (2,-1) -- (2,1) node[midway,right]
{$\calH_{02}$} -- ({1.5*2/sqrt(5)},{1.5*1/sqrt(5)}) node[dot]{};
\end{tikzpicture}
\caption{Boundaries (in bold-face $\calH_{02}$, circle points being excluded)}
\label{intro-multiD}
\end{figure}
In each domain $\calD_l$, the unknown $\w$ is governed by a specific conservation law with flux-function
$A^l=(a_i^l)_{1\leq i \leq d}:w\in\R\mapsto A^l(w)\in\R^d$:
\be
\label{num1}
 \dt \w + \sum_{i=1}^d\dxi a_i^{l}(\w) = 0,\qquad \w(t,x)\in\R,\quad t>0,\quad x\in\calD_l.
\ee

Following the description introduced in the previous section, we start focusing the discussion on the
definition of constant states \eqref{star4}-\eqref{star5}-\eqref{star6} for the global problem set on the
whole space $\R^d$. These solutions are recovered through a certain change of variable in each subdomain
$\calD_l$, for $l=0,\dots,L$,
\be
\label{chgtvarrrr}
 u(t,x) = \theta_l(w(t,x)),\quad t>0,\ x\in\calD_l,
\ee
so that the stationnary solutions $w(x)$ for the coupled problem \eqref{num1} are the real constants
$u^\star$ in the $u$ variable:
\be
\label{num3}
u(w(x))=u^\star,\quad x\in\R^d\setminus\calB.
\ee

The coupling functions $\cpl_l$ are supposed to map increasingly $\R$ onto itself and we denote once again
$\lpc_l$ the inverse functions: 
\be
 \lpc_l= \cpl_l^{-1}, \quad l=0,\ldots,L.
\ee
Observe that a different outlook where the coupling functions would be associated to the interfaces
$\calH_{kl}$ rather than to the domains themselves could only be local in space and therefore would not
allow a matching of local constant solutions so as to define a global constant solution. Here we
take advantage of the local formulation at each interface in terms of the traces of $w$, say
$\w(t,y^k)$ and $\w(t,y^l)$ on the $\calD_k$-- and $\calD_l$--side of $\calH_{kl}$, respectively 
(relatively to its normal $\nu_{kl}(y)$): 
$$
 \cpl^k(\w(t,y^k))=\cpl^l(\w(t,y^l)),\quad t>0,\ y\in\calH_{kl}.
$$

The following augmented PDE formulation is based on a vector-valued \emph{color
function} that merges the description of the coupling problem. In this multi-domain approach,
this function is based on the set of characteristic functions of each domain:
\be
\label{num4}
 v_\emptyset =\chi_{\calD_0},\quad v_1=\chi_{\calD_1}, \ldots,\quad v_L=\chi_{\calD_L},
\ee
so that the change of variable \eqref{chgtvarrrr} may be also rewritten
\be
 \label{num5}
u(t,x)=\theta_l(w(t,x)),\quad x\in\R^d\setminus\calB \textrm{ such that } v_l(x)=1.
\ee
Observe that since the $(L+1)$ domains are a partition of the whole space $\R^d$, only $L$ of the above
characteristic functions are useful to complete the coupling description of the $(L+1)$ domains. Up to
some relabeling we choose $v_1,\ldots,\v_L$, so that $v_0$ is recovered thanks to
\be
\label{num5b}
 v_\emptyset(x)=1-\sum_{l=1}^L v_l(x),\quad x\in \R^d\setminus\calB.
\ee


\subsubsection*{Multi-component coupling based on an augmented PDE's system} 

In the following we make use of the vector-valued \emph{color function} $v=(v_1,\ldots,v_L)$. At this stage,
it takes values in the discrete set $\{0\}\cup\{e_1\}\cup\ldots\cup\{e_L\}$ where $e_l$ stands for the
$l$-th canonical vector of $\R^L$. This color function is intended to be regularized and to take values in the convex hull
$
\mathbb{B}^{L}_+ = \big\{\v=(\v_1,\ldots,\v_L)\in\R^{L}\big/ \v_l\geq 0,\quad \sum_{l=1}^L \v_l \le
1\big\}
$. The problem \eqref{num1} is then understood in the augmented form (with $t>0$, $x\in\R^d$) 
\be
\label{num6}
\begin{aligned}
\displaystyle
\partial_\u \cz(\u,\v) \dt \u + \sum_{i=1}^d \partial_\u \ci(\u,\v) \dxi \u & = 0, \\
\dt \v & = 0,
\end{aligned}
\ee
where the coupling functions $\cz$ and $\ci$ are assumed to restore the formulation \eqref{num1} in terms
of $u$ in each open set $\calD_l$, that is:
\be
\label{multid-pure}
\begin{array}{ll}
\lim_{\v\to 0}\cz(\u,\v) = \gamma_0(\u), \quad &\lim_{\v\to e_l}\cz(\u,\v) = \gamma_l(\u),\\
\lim_{\v\to 0}\ci(\u,\v) = a_i^0(\gamma_0(\u)),&\lim_{\v\to e_l}\ci(\u,\v) = a_i^l(\gamma_l(\u)), \quad
1\leq i\leq d.
\end{array}
\ee
The following smoothness and monotonicity assumptions are required
\be
 \label{smoothC}
 \cz,\ci \in \mathcal{C}^2(\R\times\mathbb{B}^{L}_+),
\ee
\be
 \label{monotC0}
 \partial_\u \cz(\u,\v) > 0,\qquad u\in\R,\quad v\in \mathbb{B}^{L}_+.
\ee
This last property ensures the validity of the change of variable $\u\mapsto\cz(\u,\v)$ for any
fixed $v$, and the non--degenerate nature of the time-arrow in the augmented equations \eqref{num6}.

In this context, the augmented system in the main unknown $u$ reads
\be
\label{cauchyuv}
\begin{aligned}
\displaystyle
\dt \cz(\u,\v) + \sum_{i=1}^d \dxi \ci(\u,\v) - \sum_{i=1}^d \sum_{l=1}^L \partial_{v_l} \ci(\u,\v) \dxi
\v_l & = 0, \\
\dt \v & = 0.
\end{aligned}
\ee

In the following, it will be useful to consider the same system written in the variable $\w=\cz(\u,\v)$
(denoted by $\w(\u,\v)$, and with inverse $\u(\w,\v)$ for each fixed $\v$). Equipped with such a change of
unknown, \eqref{cauchyuv} becomes
\be
\begin{aligned}
\displaystyle
\dt \w+ \sum_{i=1}^d \dxi \ffi(\w,\v) - \sum_{i=1}^d \sum_{l=1}^{L} \lil(\w,\v) \dxi \v_l & = 0, \\
\dt \v & = 0,
\end{aligned}
\ee
where $\ffi(\w,\v)=\ci(\u(\w,\v),\v)$ and $\lil(\w,\v)=\partial_{v_l} {\ci}_{|u}(\u(\w,\v),\v)$ with
$i\in\{1,\ldots,d\}$ and $l\in\{1,\ldots,L\}$ (i.e. $\ll=\nabla_v \C$). Hereafter and to shorten the
notation, we write 
\be
\label{cauchywv}
\begin{aligned}
\displaystyle
\dt \w+ \nabla \cdot \flux(\w,\v) - \ll(\w,\v) : \nabla \v & = 0, \\
\dt \v & = 0,
\end{aligned}
\ee
with obvious notation. 

%
%

\subsubsection*{Entropy stability and well-posedness}

As already emphasized, in this work we propose a regularization mechanism based on thick interfaces that
are modeled by any suitable regularized version of the discontinuous vector-valued color function $v$
introduced in \eqref{num4}-\eqref{num5b}. For definiteness, we shall consider color functions $v$ in
$W^{2,\infty}(\R^+\times\R^d,\mathbb{B}^{L}_+)$. Obviously, it suffices to choose the initial data $v_0$
in $W^{2,\infty}(\R^d,\mathbb{B}^{L}_+)$ so as to inherit from the required smoothness in the $v$ solution
of the augmented equations \eqref{cauchywv}. In turn and arguing about this smoothness property, the
equations under consideration reduce to an inhomogeneous scalar equation in $w$:
\be
\label{withsource}
\dt \w+ \nabla \cdot \flux(\w,\v(x)) = \ll(\w,\v(x)) : \nabla \v(x),
\ee
where the right--hand side just plays the role of a classical source term; namely this term does not
contribute to the definition of the possible discontinuities of $w$. At a point of jump,
\eqref{withsource} just resumes to the classical Rankine-Hugoniot condition
\be
\label{rankinehug}
 -\sigma (w^+-w^-) + \sum_{i=1}^d \left(f_i(w^+,v)-f_i(w^-,v)\right)=0.
\ee
A selection criterion of the admissible weak solutions $w$ is of course needed, and we recast the balance law \eqref{withsource} in the main variable $u$:
\be
\label{etoile1}
 \dt  \cz(\u,\v) + \sum_{i=1}^d \partial_\u \ci(\u,\v) \dxi \u = 0
\ee
 for all smooth solutions. For such solutions, additional equations are deduced and based on any (strictly) convex function $\varpi\mapsto\U(\varpi)$, by multiplying \eqref{etoile1}
by $\U'(\cz(\u,\v))$, 
\be
\label{etoile2}
\dt \U(\cz(\u,\v)) + \sum_{i=1}^d \partial_u \Q_i(\u,\v)\partial_{x_i} u = 0,
\ee
where
\be
\label{entropyflux}
\Q_i(\u,\v) = \int^\u \U'(\cz(\theta,\v))\partial_\theta \ci(\theta,\v) d\theta, \quad 1\le i \le d.
\ee
We thus get from \eqref{etoile2} the equivalent form for smooth solutions $u$:
\be
\label{etoile4}
\dt \U(\cz(\u,\v)) + \sum_{i=1}^d \dxi \Q_i(\u,\v) = \sum_{i=1}^d\sum_{l=1}^L \partial_{v_l} \Q_i(\u,\v)
\dxi \v_l.
\ee
Observe that the above right--hand side is nothing but a classical source term since we again emphasize
that the color function $v$ is smooth. As a consequence, the weak form of \eqref{etoile4} for
discontinuous solutions $u$ reads:
\be
\label{entropyu}
\dt \U(\cz(\u,\v)) + \sum_{i=1}^d \dxi \Q_i(\u,\v) \le \sum_{i=1}^d\sum_{l=1}^L \partial_{v_l} \Q_i(\u,\v)
\dxi \v_l,
\ee
which naturally plays the role of an (inhomogenous) entropy inequality for selecting the relevant weak
solutions. Hereafter, we shall make use of the inequalities \eqref{entropyu} for all convex entropy $\U$.
These will be alternatively invoked (essentially when the color function is locally constant) in the $w$
variable:
\be
\label{etoile6}
\dt \U(\w) + \sum_{i=1}^d \dxi \F_i(\w,\v) - \sum_{i=1}^d\sum_{l=1}^L  \Lil(\w,\v) \dxi \v_l  \le 0,
\ee
with
\be
\label{entropyfluxw}
\quad \F_i(\w,\v) = \Q_i(\u(\w,\v),\v), \qquad \L_i(\w,\v) = \partial_v {\Q_i}_{|_u}(\u(\w,\v),\v), 
\qquad 1\le
i\le d.
\ee
To shorten the notation, equation \eqref{etoile6} are written as 
\be
\label{entropyw}
\dt \U(\w) + \nabla \cdot \F(\w,\v) -  \L(\w,\v) : \nabla \v  \le 0.
\ee
The inhomogeneous scalar conservation law \eqref{withsource} supplemented with all the entropy
inequalities \eqref{etoile6} naturally falls within Kruzkov 's theory of entropy
solutions, since the color function $v$ belongs to $W^{2,\infty}(\R^d,\mathbb{B}^{L}_+)$. 
Therefore, Kruzkov's uniqueness theorem for scalar conservation law with smooth inhomogeneities applies and asserts the uniqueness of the entropy weak solution of the Cauchy problem
\eqref{withsource}-\eqref{etoile6} with initial data $w_0\in L^1(\R^d)\cap L^\infty(\R^d)$.

Hereafter, we shall prove existence and uniqueness of a solution to the coupled problem
\eqref{withsource}-\eqref{etoile6} thanks to a multidimensional well-balanced finite volume method
formulated on general triangulations. Here, the well-balanced property means that the solutions in the $u$
variable is kept constant in time and space as soon as the initial data $u_0$ is chosen constant
whatever the definition of the (smoothly varying in space) color function $v$. This well-balanced
property is obviously a constancy property of primary importance.


\section{A well-balanced finite volume scheme for coupling problems}

\subsection{Terminology and assumptions}

Before stating our main result, we introduce some notation and motivate our formulation of the
finite volume method under consideration. To meet the well--balancing property, the finite volume framework we develop uses 
two
families of triangulations. The first triangulation, denoted by $\Th$, is made of general polyhedra and
will be referred to as the primal mesh. Then a closely related triangulation is of concern, the 
dual mesh $\Thstar$, whose polyhedra are derived from the edges of the primal one. As we shall see, dual
meshes may not uniquely defined from $\Th$ and it will turn that a given choice essentially affects the
closed-form of expression of the CFL restriction in the  (time explicit) finite volume method.

Equipped with these primal and dual meshes, approximate solutions $\u_h$ and $\v_h$ of the Cauchy problem
(\ref{cauchyuv}) with initial data $(\u_0,\v_0)$, are sought as piecewise
constant functions. In constrast with the usual approach, constant values for $\uh$ and $\vh$ will
not be co--localized: $\uh$ (and $\vh$, respectively) will assume constant values in each polyhedron of the
primal mesh (and the dual mesh, resp.).

To facilitate the derivation of the proposed well-balanced scheme, we shall take advantage of the
regularity of the color function $\v$, which provides some room for the specific definition of
the discrete approximation $\vh$: it may range from a local averaged form to a point-wise evaluation.
Without real loss of generality, we use an average value of $\v$ along
each edge of the primal mesh. This choice allows to bypass the definition
of the dual mesh from the edges of the primal one: a convex sequence of reals, in turn, provide 
sufficient information on the dual mesh. On the ground of this observation, we shall give a first brief but
sustained mathematical presentation of the finite volume method under consideration. We shall then be in a
position to state the main result of this paper. At last, we shall close this section with a comprehensive
construction of the proposed finite volume approximation when deriving dual meshes from the primal one.

The primal mesh, $\Th$, is a general (locally finite) triangulation of $\R^d$ made of 
non--overlapping, non--empty, and open polyhedra : $\cup_{K\in\Th}\overline{K} = \R^d$. We assume that for
every pair of distinct polyhedra $K,K'\in\Th$ the set $K\cap K'$ is either an edge $e$ of both $K$ and
$K'$ or a set with Haussdorf dimension less
than or equal to $d-2$. The set of edges of a polyhedron $K$ is denoted by $\partial K$; and for each
$e\in\partial K$, $\nuek\in\R^d$ represents the outward unit normal vector to the edge $e$ (see
Figure~\ref{figMesh}). The volume of $K$ and the $(d-1)$-measure of $e$ are
denoted $\volk$ and $\vole$, respectively. Given an edge $e$ in $K$, $K_e$ denotes the unique polyhedron
in $\Th$ that
 shares the same edge $e$ with $K$. We set 
$
h = \sup_{K\in\Th} h_K,
$  
where $h_K$ is the exterior perimeter of the polyhedron $K$, and 
assume that the triangulation $\Th$ satisfies the following non degeneracy condition 
\be
\label{nondegenere}
\sup_K \frac{h_K \, \pk}{\volk} \leq C
\ee
for some constant $C>0$. Here, $\pk$ denotes the perimeter of $K$ defined by 
$\pk = \sum_{e\in \partial K} \vole$. 

It is unnecessary, at this stage,  to provide a comprehensive derivation of the {\sl dual mesh} $\Thstar$
that one could define  
from the edges $e$ in the primal mesh $\Th$. Recall that, by design, a dual mesh is made of 
non--overlapping, non--empty, and open polyhedra denoted by $\ks(e)$ with $\displaystyle \cup_{e\in \Th}
\ks(e) = \R^d$.
By construction, both sets $\ks(e)\cap K$ and $\ks(e)\cap K_e$ are non--empty for all pair $(K,K_e)$ of
adjacent polyhedra parametrized by the edges $e$ in $\Th$. Note that the set $\ks(e)\cap K$ is a subcell
of $K$. Then, the only information
about $\Thstar$ that is required in this section is
a given convex sequence of reals prescribed in each polyhedron $K$ in $\Th$; 
we denote by $\{\ake\}_{\{e,e\in \partial K\}}$, that satisfies (for any $K$ in $\Th$)
\be
\label{defake}
0 < \ake < 1 \qquad (e\in\partial K), \qquad \quad \sumek \ake = 1.
\ee
We will see later that the coefficient $\ake$ is nothing but the ratio of the volume of $\ks(e)\cap K$ 
to the volume of $K$, where $\ks(e)$ stands for the dual polyhedron of $K$ attached to any edge $e$ in
$\partial K$:
\be
\label{intake}
\ake = \frac{\vert \ks(e)\cap K\vert}{\volk},\qquad e\in\partial K.
\ee
At last, the time increment, denoted by $\tau$, is assumed to satisfy
$\frac{\tau}{h} \leq C$
and the primal mesh to be constrained by
\be
\label{ehnondeg}
C_1 \leq \frac{|e|}{h} \leq C_2
\ee
for some constants $C, C_1, C_2>0$. Whereas the latter is probably not an optimal condition, it
sufficies to ensure the non degeneracy of the mesh: all one-dimensional characteristic lengths are of
order $h$. A key property for the forthcoming CFL condition, is that under these assumptions the area
$\vert\ks(e)\cap K\vert$ is not smaller than $O(h^2)$: there exists a positive constant $c$ such
that
\be
\label{dualnondegen}
 ch^2 \le \vert\ks(e)\cap K\vert.
\ee

We use the notation $t^n = n\tau$. As already underlined, we will seek at each time level $t^n$
approximate solutions $\u_h$ and $\v_h$ of the Cauchy problem (\ref{cauchyuv}) with initial data
$(\u_0,\v_0)$, under the form of piecewise constant functions with:
\be
\left.
\begin{array}{lll}
&\uh(x,t^n) = u^n_K, \quad x\in K,\quad K\in\Th, \\~\\
&\vh(x,t^n) = \vh(x) = v_e, \quad x\in \ks(e), \quad e\in\Th.
\end{array}
\right.
\ee
Here and since the solution $\v$ in the Cauchy problem (\ref{cauchyuv}) does not depend on time, it seems
natural to set $\vh(x,t^n) = \v(x) = \vh^0(x)\in\R^L$ for all time level $t^n$, for some discrete
approximation $\vh^0$ of the smooth function $\v_0$. We introduce 
\be
\label{defve}
\vh(x) = \v_e =  \displaystyle~\frac{1}{\vole} \int_e \v_0(y) dy, \quad x\in \ks(e), \quad e\in \Th,\ee
while the discrete version of the possibly discontinuous initial data $\u_0$ is chosen according to the
usual full averaging procedure over each polyhedron $K$:
\be
\label{defukinit}
\u_h^0(x) = \u^0_K = \displaystyle \frac{1}{\volk}\int_K \u_0(y) dy, \quad x\in K, \quad K\in\Th.
\ee

\begin{remark} 
\label{gradve}
Since $v_0$ is regular, any other consistent definition for the constant value $\v_e$
in ${\ks}(e)$ would have been relevant. The interest in the particular choice (\ref{defve}) stems from the
following Green formula, valid for each polygonal domain $K$:
\begin{equation*}
X \, \sumek \ve_l \nuek \vole = \int_K \nabla \cdot (v_l(x) \, X) dx = X \, \int_K \nabla
v_l(x) dx,
\end{equation*}
where $X$ denotes any fixed vector in $\R^d$ and $\ve_l$ (and $v_l$, respectively) the $l$-th component
of the vector $\ve\in\R^L$ (and $v$, resp.).
Hence the proposed average value in (\ref{defve}) comes with the identity:
$
\int_K \nabla v_l(x) dx =\sumek \ve_l \nuek \vole.
$
In a tensorial notation, we thus get $
\int_K \nabla v(x) dx =\sumek \ve \otimes \nuek \vole.
$
\end{remark}

The evolution in time of the discrete solution $\uh$ will rely on a family of  
{\sl numerical flux-functions,} associated with each edge $e$ of any polyhedron $K$ in $\Th$.
Besides other properties, these numerical flux functions must meet some consistency property with the
exact equation for governing $\u$ in (\ref{cauchywv}), namely:
\be
\label{equk}
\dt \w(\u,\v)+ \nabla \cdot \flux(\w(\u,\v),\v) - \ll(\w(\u,\v),\v) : \nabla \v = 0, \qquad x\in K, \quad
t\in (t^n,t^{n+1}).
\ee
Observe that in the neighborhood $\ks(e)$ of each edge $e$, where $\vh$ reduces to a
constant value $\v_e$, the above equation boils down to the scalar equation in the unknown $\w =
\w(\u,\ve)$:
\be
\label{eque}
\dt \w+  \nabla \cdot \flux(\w,\v_e) = 0, \qquad x\in \ks(e)\cap K, \quad t\in (t^n,t^{n+1}).
\ee
This in turn leads us to define the required numerical flux function at each edge $e$ in $\Th$ as a
locally Lipschitz continuous two-point flux-function $\gek(.,.;\ve) : \R\times\R \to \R$ that satisfies
the consistency property:
\be
\label{gekconsi}
\gek(\w,\w;\ve) = \flux(\w,\ve)\cdot\nuek,
\ee
the conservation property:
\be
\label{gekconse}
\gek(\w,\w_e;\ve) = - g_{e,K_e}(\w_e,\w;\ve), 
\ee
for all reals $\w$ and $\w_e$, and the monotonicity property 
\be
\displaystyle
\label{gekmonot}
\frac{\partial g(\w,\w_e;\ve)}{\partial w} \ge 0, \qquad \frac{\partial g(\w,\w_e;\ve)}{\partial \w_e} \le
0.
\ee
In addition, we assume that the numerical flux depend (locally) Lipschitz continuously in the variable
$v_e$.

Standard $3$--point monotone schemes in the scalar framework obey (\ref{gekconsi})--(\ref{gekmonot}) and
 that the main results in this paper are easily extended to all E-schemes (Osher \cite{Osher84}). 
For clarity,   the dependence in the parameter $\ve$ appears explicitly in the numerical flux-function $\gek(.,.;\ve)$.

\begin{remark}
\label{lipgek}
Since the function $g(.,.;.) : \R^3\to \R$ is locally Lipschitz continuous in its three
arguments, for all compact ${\cal K}\subset \R^3$, there exists some positive constant ${\cal C}_{\cal K}$
such that for all triple $(\w^{(1)},\w^{(1)}_e,\v^{(1)}_e)$ and $(\w^{(2)},\w^{(2)}_e,\v^{(2)}_e)$ in
${\cal K}$, the following estimate holds true:
$$
\left.
\begin{aligned}
&\vert \gek(\w^{(2)},\w^{(2)}_e;\v^{(2)}_e)-~\gek(\w^{(1)},\w^{(1)}_e;\v^{(1)}_e)\vert  
 \le {\cal C}_{\cal K}
\, \big(\vert \w^{(2)}-\w^{(1)}\vert+\vert \w^{(2)}_e-\w^{(1)}_e\vert+\vert
\v^{(2)}_e-\v^{(1)}_e\vert
\big).
\end{aligned}
\right.
$$
\end{remark} 


\subsection{Definition of the well-balanced scheme}

We are now in a position to define the finite volume approximation of (\ref{equk}).
Assuming that the approximate solution $\uh(.,t^n)$ is known at time $t^n$, we determine the 
evolution up to the next time level $t^{n+1}$ as follows: 

\begin{description}
\item {\it Subcell reconstruction.} At each time $t^n$ in each polyhedron $K$ of $\Th$, we consider any edge $e \in \partial K$ and introduce the subcell state 
\be
\label{defwkve}
\wkve = \cz(\uk, \ve), \quad e\in \partial K,
\ee
as well the following average over all edges of $K$ 
\be
\label{defwk}
\displaystyle
\wk = \sumek \ake \wkve.
\ee

\item {\it Evolution in time}. In order to the discrete solution $\uh$ at time $t^{n+1}$, we define (in each polyhedron $K$) $ \uuk$ to be the
unique solution of 
\be
\label{defuuk}
\displaystyle
 \sumek \ake ~\cz(\uuk,\ve) = \wwk,
\ee
where the state $\wwk$ is given by the finite volume scheme 
\be
\label{defwwk}
\wwk = \wk - \frac{\tau}{\volk}\sumek \gek(\wkve,\wkeve;\ve)\vole + \frac{\tau}{\volk}\sumek
\flux(\wkve,\ve)\cdot \nuek \vole.
\ee
\end{description}

This completes the description of our numerical method. The proposed finite volume method is explicit in time and, for the sake of stability, we need to impose a CFL (Courant, Friedrichs, Lewy) condition which reads, for all polyhedra $K$ in $\Th$ and edges $e\in\partial K$, 
\be
\label{CFL}
\frac{\tau}{\volk}~\frac{\vole}{\ake} \sup_{\u\in [m,M]}\Big\vert \frac{\partial
\flux(\w(\u,\ve),\ve)}{\partial \w}\Big\vert \le 1,
\ee
where $\displaystyle m =\inf_{x\in\R^d} \u_0(x)$ and $\displaystyle M=\sup_{x\in\R^d} \u_0(x)$.

Due to the dimensional hypothesis \eqref{defake}-\eqref{ehnondeg}-\eqref{dualnondegen} the ratio $\volk\ake/\vole$ satisfies
\begin{equation*}
 \volk\frac{\ake}{\vole} = \frac{\vert\ks(e)\cap K\vert}{\vole} \ge \frac{c}{C_2}h,
\end{equation*}
so that the CFL condition can not imply the degeneracy of the time step $\tau$, that decreases at most as
$O(h)$. We will see in Section~\ref{FVSCHEME} how to build suitable primal and dual meshes.

Several comments are in order. First observe that the constitutive assumptions
(\ref{smoothC})--(\ref{monotC0}) on the coupling function $\cz(.,.)$ immediately yields existence and
uniqueness of a solution to the nonlinear equation (\ref{defwkve}) so that the finite volume method
(\ref{defwkve})--(\ref{defwwk}) is well defined. The formulas (\ref{defwkve}) and (\ref{defuuk})
obviously express the same identity at the times $t^n$ and $t^{n+1}$, and are redundant: the finite volume method essentially reduces to
(\ref{defuuk})--(\ref{defwwk}). As they stand, they nevertheless ease the description of the method.

Next, it is worth observing that the consistency condition (\ref{gekconsi}) allows in (\ref{defwwk}) to
recast the flux balance $\sumek \flux(\wkve,\ve)\nuek\vole$ as $\sumek\gek(\wkve,\wkve;\ve)\vole$. Here we
stress that at each edge $e$ in $\partial K$, both the numerical flux-function $\gek(\wkve,\wkeve;\ve)$
and its counterpart $\flux(\wkve,\ve)\cdot \nuek$ are evaluated thanks to the subcell values $\wkve$
(\ref{defwkve}) and not to their averaged form $\wk$ in (\ref{defwk}). The motivation is twofold. In a
first hand, the two flux balances involved in (\ref{defwwk}), namely $\sumek \gek \vole$ and $\sumek
\flux(\wkve,\ve)\cdot\nuek \vole$, make the proposed formula to be a consistent finite volume
approximation of the exact equation (\ref{equk}) for governing $\u$: namely, the first one will be seen
hereafter to be consistent with $\nabla \cdot \flux(\w,\v)$ while the second one actually provides a
consistent approximation of the source term $\ll(w,v) : \nabla v$. In a second hand, the 
discretization of the source term is seen to be well--balanced.

\begin{proposition}[Well-balanced property] 
\label{lemwb}
When the initial data $\u_0$ for (\ref{equk}) is a constant function $\u_0(x) = \u^\star (x\in\R^d)$,  
then, for any choice of the color function $\v$ in (\ref{equk}), the
discrete solution $\uh$ of (\ref{defwkve})--(\ref{defwwk}) is also constant, with 
\be
\uh(x,t^n) = \u_0(x) = \u^\star, \quad x\in\R^d
\ee
for all time level $t^n$. 
\end{proposition}

In other words, the finite volume method (\ref{defwkve})--(\ref{defwwk}) is well-balanced with respect to
all the natural equilibria of (\ref{equk}).

\begin{proof}
The discrete initial data (\ref{defukinit}) clearly reads $\u^0_h(x)= \u^\star$ for all $x$ in $\R^d$ so
that at the first subcell reconstruction step, we get $\w^0_{K,e} = \cz(\u^\star,\ve) = \w^0_{K_e,e}$ for
any edge $e$ of an arbitrary polyhedron $K$ in $\Th$. Consequently, the numerical flux
$\gek(\w^0_{K,e},\w^0_{K_e,e};\ve)$ at any edge $e$ boils down to $\flux(\w^0_{K,e},\ve)\cdot \nuek$
in view of the consistency condition (\ref{gekconsi}). Namely the two flux balances in the updating
formula (\ref{defwwk}) cancel out and we end up with $\w^1_K = \w^0_K = \sumek \ake \cz(\u^\star,\ve)$
thanks to the definition (\ref{defwk}). Arguing about uniqueness, we thus get when solving (\ref{defuuk})
$\u^1_K = \u^\star$ for any polyhedron $K$ of $\Th$: namely $\uh(x,t^1)=\u^\star$ for all $x$ in
$\R^d$. An immediate recursion extends the result to the subsequent time levels.
\end{proof}

To conclude this paragraph, it is worth illustrating that the last flux-balance entering the finite volume
approximation (\ref{defwwk}) actually provides a consistent approximation of the source term $\ll(w,v):
\nabla v$. For the sake of simplicity, we temporarily adopt (cf.~Remark~\ref{example2}):
$$
\left.
\begin{array}{lll}
\cz(\u,\v) =  (1-v) \gamma_-(\u) + v \gamma_+(\u), \\
\ci(\u,\v) = (1-v) a^-(\gamma_-(\u)) + v a^+(\gamma_+(\u)), \quad 1 \le i \le d,
\end{array}
\right.
$$
so that $\flux(w,v)$ and $\ll(w,v)$ in (\ref{equk})  read
$
\flux(\w(\u,\v),\v) =  (1-v) A^-(\gamma_-(\u)) + v A^+(\gamma_+(\u)),
$
and
$
\ll(\w(\u,\v),v)= \Big(A^+(\gamma_+(\u)) - A^-(\gamma_-(\u))\Big).
$
It can be then readily computed:
$$
\left.
\begin{aligned}
\displaystyle
& \sumek \flux(\wkve,\ve)\cdot \nuek \vole 
\\
& =  
\displaystyle
\sumek \Big((1-\ve)A^-(\gamma_-(\uk)) + \ve A^+(\gamma_+(\uk))\Big) \cdot \nuek \vole
\\
& = 
\displaystyle
\Big(A^+(\gamma_+(\uk)) - A^-(\gamma_-(\uk))\Big) \cdot \sumek  \ve \vole \nuek
+ A^-(\gamma_-(\uk))\cdot \Big(\sumek  \vole \nuek\Big)
\end{aligned}
\right.
$$
which is nothing else a consistent discretization of $\ll(\w(\u,\v)): \nabla v$, in view of the
representation formula in Remark~\ref{gradve} (for $\nabla v$) and the identity $\displaystyle \sumek
\vole\nuek= 0$.

These straightforward calculations allows to bridge the finite volume formula (\ref{defwwk}) to the
governing equation (\ref{equk}) for $\w(\u,\v)$, expressed over $K$, namely where $\vh$ does achieve
distinct values. The gap in between (\ref{equk}) and its reduced version (\ref{eque}) ({\it i.e.} with
$x\in \ks(e)\cap K$) will be definitely closed when revisiting the finite volume approximation
(\ref{defwkve})--(\ref{defwwk}) with primal--dual meshes (in Section \ref{FVSCHEME}).


\subsection{Main convergence result}

We are now in a position to state the main result of this paper. 

\begin{theorem}[Well--balanced finite volume method for multi--dimensional coupling problems] 
\label{maintheo}
Consider the Cauchy problem (\ref{cauchyuv})--(\ref{entropyu}) with initial data $\u_0\in L^\infty(\R^d)$
and $\v_0\in W^{2,\infty}(\R^d)$
under the constitutive assumptions (\ref{smoothC})--(\ref{monotC0}).
Let $\uh$ be the sequence of approximate solutions defined by the finite volume method
(\ref{defve})--(\ref{defukinit}) and (\ref{defwkve})--(\ref{defwwk}) with numerical flux-functions satisfying
the conditions (\ref{gekconsi})--(\ref{gekmonot}). 
Then under the CFL restriction (\ref{CFL}), the sequence $\uh$ is uniformly bounded in
$L^\infty(\R_+\times \R^d)$ and converges (when $h\to 0$) in the
$L^p_{loc}$ norm strongly ($1\le p < \infty$) to the unique entropy
solution $\u$ to the problem (\ref{cauchyuv})--(\ref{entropyu}): namely for all time $T>0$ and for all
compact ${\cal K}$ in $\R^d$
$$
\lim_{h \to 0} 
\vert\vert \u - \uh\vert\vert_{L^p((0,T)\times{\cal K})} =0.
$$ 
\end{theorem}

The rest of this paper is devoted to a proof of this theorem.


\section{Finite volume approximations on primal-dual meshes} 
\label{FVSCHEME}

\subsection{A convex combination}

One of our objectives in this section is explaining how the coefficients $\alpha_{K,e}$ should be
determined. Arguing about the formula-definitions (\ref{defwkve})--(\ref{defwk}) at time $t^n$ and the
consistency condition (\ref{gekconsi}), we obtain the following statement.

\begin{lemma}[Edge values and convex combination] 
\label{lemwwkve}
For any  polyhedron $K$ of $\Th$ and edge $e$ in $\partial K$, let us define the following subcell
states:
\be
\label{defwwkve}
\wwkve = \wkve - \frac{\vole}{\ake}\frac{\tau }{\volk}\Big(\gek(\wkve,\wkeve;\ve) -
\gek(\wkve,\wkve;\ve)\Big).
\ee
Then $\wwk$ in (\ref{defwwk}) are recovered by the following averaging procedure:
\be
\label{convwwk}
\wwk = \sumek \ake \wwkve.
\ee
\end{lemma}

Observe that the finite volume formula (\ref{defwwkve}) for $\wwkve$ is nothing but a consistent
approximation of the one dimensional conservation law:
$
\dt \w+  \nabla \cdot \flux(\w,\v_e) = 0. 
$
The reason for calling $\wwkve$ a subcell state will be explained in this paragraph and is 
at the core of the re-interpretation of the finite volume
formula (\ref{defwwk}) with primal--dual meshes.

To further proceed, let us underline that the identity (\ref{convwwk}) just expresses that $\wwk$ actually
is a convex decomposition of the subcell states $\wwkve$. When understood in their quasi-one
dimensional form (\ref{defwwkve}), the latter can be recognized as extensions to the present
inhomogenous setting of partial states entering similar convex decompositions that have proved well
suited in the analysis of homogeneous multidimensional finite volume methods
 \cite{CoquelLeFloch90,CoquelLeFloch93}. Indeed, the interest in such a convex
decomposition primary stems from the fact that many of the basic stability properties satisfied by the
scheme (\ref{defwwkve}) in one space variable are right away inherited in several space variables thanks
to convexity under some CFL restriction. Observe that the relevant CFL condition for 
(\ref{defwwkve}) reads 
\be
\label{cflloc}
\frac{\tau}{\volk}~\frac{\vole}{\ake} \Big\vert \frac{\gek(\wkve,\wkeve;\ve) -
\gek(\wkve,\wkve;\ve)}{\wkeve-\wkve}\Big\vert \le 1,
\ee
and hence the CFL restriction (\ref{CFL}).

At last and arguing about the definition (\ref{defwwkve}), the subcell reconstruction step (\ref{defwk})
at time $t^{n+1}$ and the formula (\ref{defuuk}), we deduce the (seemingly trivial) identities
\be
\label{consrec}
\sumek \ake \w^{n+1}_{K,e}=\wwk= \sumek \ake \wwkve.
\ee
In other words, all the steps involved in the method are locally conservative: this natural property will
play a central role in the forthcoming analysis.


\subsection{A reformulation of the scheme}

The derivation of a dual mesh $\Thstar$ from the edge of the primal one $\Th$ may be performed as follow.
For any (open) polyhedron $K$, the idea is to pick an internal node $x_K$ in $K$ which choice is
left arbitrary at this stage. Such a procedure is given below a systematic definition independent of the
mesh refinement $h$. Equipped with the node $x_K$, we define for any edge $e$ in $K$ the convex hull
of $e$ and $x_K$. The interior of this convex hull, we denote by ${\cal E}(x_K,e)$, yields a non--empty
open polyhedron made of $(d+1)$ edges. Observe that the following properties are met by construction: for
any pair of distinct edges $e,e'$ in $\partial K$ with $K$ an arbitrary polyhedron in $\Th$
\be
{\cal E}(x_K,e)\cap K = {\cal E}(x_K,e),\quad {\cal E}(x_K,e)\cap{\cal E}(x_K,e')=\emptyset, 
\ee 
while 
$
\sumek  {\cal E}(x_K,e) = K.
$
Then, the required definition of the polyhedron $\ks(e)$ of the dual mesh $\Thstar$, attached to any
edge $e$ in $\Th$ with adjacent polyhedron $K$ and $K_e$, follows from 
\be
\label{defkse}
\ks(e) = {\cal E}(x_K,e)\cup {\cal E}(x_{K_e},e).
\ee
We refer the reader to Figure \ref{figMesh} for an illustration.

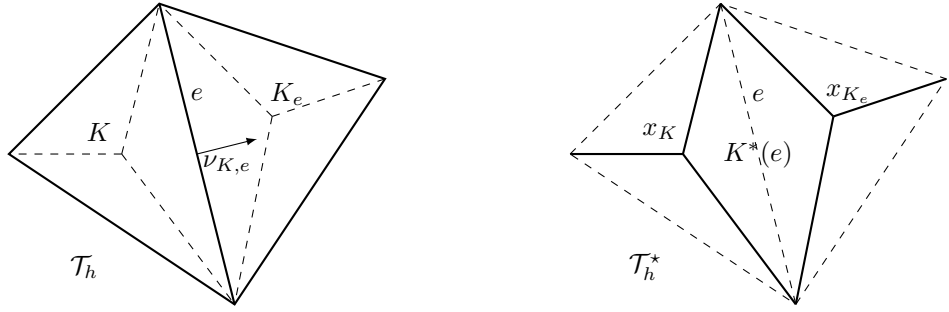
\begin{figure}[!ht]
\centering
\hfill
{
 \begin{tikzpicture}[scale=2.0,>=latex]
  \begin{scope}[thick]
  \draw (0,0) -- (1,1) -- (2.5,0.5) -- (1.5,-1) -- cycle;
  \draw (1,1) -- (1.5,-1) node[midway,above,inner sep=20pt] {$e$};
  \node at (0.60,0.15) {$K$};
  \node at (1.85,0.40) {$K_e$};
  \end{scope}
  \begin{scope}[dashed]
  \draw (0.75,0) -- (0,0);
  \draw (0.75,0) -- (1,1);
  \draw (0.75,0) -- (1.5,-1);
  \draw (1.75,0.25) -- (1,1);
  \draw (1.75,0.25) -- (1.5,-1);
  \draw (1.75,0.25) -- (2.5,0.5);
  \end{scope}
 \node at (0.5,-0.75) {$\Th$};
 \draw[->] (1.25,0) -- +(0.4,0.10) node[midway,below] {$\nuek$};
 \end{tikzpicture}%
}%
\hfill%
{
\begin{tikzpicture}[scale=2.0]
  \begin{scope}[dashed]
  \draw (0,0) -- (1,1) -- (2.5,0.5) -- (1.5,-1) -- cycle;
  \draw (1,1) -- (1.5,-1) node[midway,above,inner sep=20pt] {$e$};
  \end{scope}
  \begin{scope}[thick]
  \draw (0.75,0) -- (0,0);
  \draw (0.75,0) -- (1,1);
  \draw (0.75,0) -- (1.5,-1);
  \draw (1.75,0.25) -- (1,1);
  \draw (1.75,0.25) -- (1.5,-1);
  \draw (1.75,0.25) -- (2.5,0.5);
  \node at (0.60,0.15)  {$x_K$};
  \node at (1.85,0.40) {$x_{K_e}$};
  \end{scope}
  \node at (1.25,0) {$K^*(e)$};
 \node at (0.5,-0.75) {$\Th^\star$};
 \end{tikzpicture}%
}%
\hfill
\caption{Primal and dual meshes, edges and vertices.}
\label{figMesh}
\end{figure}
The constructive procedure for defining the internal node $x_K$ independently of $h$ relies on the set of
vertices $\vartheta$ of the polyhedron $K$, together with a convex sequence of reals
$\{\beta_{K,\vartheta}\}_{\{\vartheta,\vartheta\in K\}}$ satisfying:
$$
0<\beta_{K,\vartheta}< 1, \quad \vartheta\in K;\qquad \sum_{\vartheta\in K} \beta_{K,\vartheta} = 1.
$$
The required internal node $x_K$ in $K$ is then defined by its coordinates in $\R^d$:
$
x_K = \sum_{\vartheta\in K} \beta_{K,\vartheta}~ x_\vartheta,
$ 
where $x_\vartheta$ stands for the coordinates of the vertex $\vartheta$. This construction ensures the correct behavior of the primal and dual meshes with the
definition of the $\ake$ and with the previous non--degeneracy assumptions
(\ref{intake})-(\ref{dualnondegen}), the CFL condition (\ref{CFL}) is then only modified according to the
choice of the function $v$ and its discrete representation.

To further proceed in the comprehensive derivation of the finite volume framework, some additional
notation is in order. For any $K$ in $\Th$ and $e$ in $\partial K$, an edge of a dual polyhedron
$\ks(e)\in\Thstar$ or of the subcell $\ks(e)\cap K$ of $K$ will be indifferently denoted by $\es$. Observe
that with little abuse in the notation, an edge $e$ of some cell $K$ of the primal mesh $\Th$ is also a
dual edge of the subcell $\ks(e)\cap K$: see indeed Figure \ref{figMesh}. At last $\nuesk \in\R^d$ stands
for the outward unit vector normal to the edge $\es$.

Equipped with these notation, we are in a position to re-interpret the quasi-one dimensional state
$\wwkve$ introduced in (\ref{defwwkve}) in term of a state in the subcell $\ks(e)\cap K$ of $K$, thanks
to the following simple but key identity:
$$
\sumesdk \voles \nuesk = 0, \quad \hbox{i.e.} \quad \vole \nuek = -\sum_{\es\in \ks(e)\cap K, ~\es\not =
e} \voles\nuesk.
$$
It is then straightforward to recast $\wwkve$ according to:
\be
\left.
\begin{aligned}
\wwkve &= \wkve - \frac{\tau}{\ake \volk} \gek(\wkve,\wkeve;\ve) \vole + \frac{\tau}{\ake \volk}
\flux(\wke)\cdot\nuek \vole,\\
& = \wkve - \frac{\tau}{\vert \ks(e)\cap K\vert}\Big( \gek(\wkve,\wkeve;\ve)\vole \\
&\quad + \displaystyle\sum_{\es\in \ks(e)\cap K, ~\es\not = e} \flux(\wkve,\ve)\cdot \nuesk
\voles\Big),
\end{aligned}
\right.
\ee
where we have used the interpretation (\ref{intake}) of $\ake$. Introducing the numerical flux
formula:
\be
\label{fluxstar}
g_{\es, \ks(e)} = 
\begin{cases}
\gek(\wkve,\wkeve;\ve), &\hbox{if } \es = e; \\  
\flux(\wkve,\ve)\cdot\nuesk, &\hbox{otherwise},
\end{cases}
\ee
$\wwkve$ thus reads 
\be
\label{balstar}
\wwkve = \wkve - \frac{\tau}{\vert \ks(e)\cap K\vert}\sumesdk g_{\es,\ks(e)} \voles. 
\ee 
We can clarify the origin of the definition
$g_{\es, \ks(e)} = \flux(\wkve,\ve)\cdot\nuesk$ for edges $\es$ distinct from $e$. For such an edge $\es$,
it is worth introducing the adjacent subcell~$\ks(e')\cap K$ to $\ks(e)\cap K$ in $K$: {\it i.e.} with
$e'$ in $\partial K$ such that $\ks(e')\cap\ks(e) = \es$. Note that $\es$ is of course distinct from $e'$.
We then successively rewrite the left-- and right--hand numerical flux at $\es$, say $g_{\es,
\ks(e)}$ (respectively $g_{\es, \ks(e')}$), as follows:
$$
\flux(\w(\uk,\ve),\ve)\cdot \nuesk, \quad \hbox{respectively :} -\flux(\w(\uk,\ve'),\ve')\cdot \nuesk,
$$ 
since by definition (\ref{defwkve}) $\wkve = \w(\uk,\ve)$ and $\w_{K,e'} = \w(\uk,\ve')$ and,
equivalently,
\be
\label{godflux}
\left.
\begin{array}{lll}
&\Big(\flux(\w(\u,\v),\v)\cdot \nuesk\Big)(\omega(0-)), \\
& \qquad \qquad  \hbox{respectively :} -\Big(\flux(\w(\u,\v),\v)\cdot \nuesk\Big)(\omega(0+)),
\end{array}
\right.
\ee
where $\omega(0^\mp)$ stands for the left and right traces at $\xi = 0$ of the self-similar function
$\omega : \xi \in \R_\xi \to (\u(\xi),\v(\xi))\in \R\times\R^L$ given by
\be
\label{defauto}
\omega(\xi) =
\begin{cases}
 (\uk, \ve),\quad &\xi <0,\\
 (\uk, \ve'),\quad &\xi >0.\\
\end{cases}
\ee
From Section \ref{formulation}, recall 
that the Riemann solution of 
\be
\label{defriem}
\begin{aligned}
\partial_t w + \partial_x \Big(\flux(\w(\u,\v),\v)\cdot \nuesk\Big)(\u,\v) - \partial_v
\Big(\flux(\w(\u,\v),\v)\cdot \nuesk\Big) : \nabla v &= 0,\\
\partial_t v &= 0
\end{aligned}
\ee
(with initial data $( (\uk, \ve),~x <0,~~ (\uk, \ve'),~x >0)$)
consists in a standing wave separating $
(\uk, \ve)$ from $ (\uk, \ve')$, and thus coincides with $\omega(\xi)$ in (\ref{defauto}). It is therefore
clear that the flux--functions in (\ref{godflux}) actually results from the Godunov method applied to the augmented
system (\ref{defriem}) at the edge $\es$. In other terms, the finite volume formula
(\ref{fluxstar})--(\ref{balstar}) in each subcell $\ks(e)\cap K$ may be understood as an approximation of
the balance law for governing $\u$ in (\ref{equk}):
$$
\dt \w(\u,\v)+ \nabla \cdot \flux(\w(\u,\v),\v) - \ll(\w(\u,\v),\v) : \nabla \v = 0, \quad x\in K,
~~t\in(t^n,t^{n+1}).
$$
This interpretation closes the gap in between the governing equation (\ref{equk}) for $\u$ and its reduced
form (\ref{eque}) expressed in $\w$:
$$
\dt \w+  \nabla \cdot \flux(\w,\ve) = 0, \quad x\in \ks(e)\cap K, ~~t\in(t^n,t^{n+1}).
$$


\subsection{Sup-norm estimates}

Throughout the upcoming sections, the assumptions of Theorem \ref{maintheo} are tacitly assumed to be
valid. Their formulations are thus skipped over in any forthcoming statements. The main result of this
section ensures that the sequence of approximate solutions $\uh$ stays uniformly bounded in
$L^{\infty}(\R_+\times\R^d)$ as a consequence of the following result.

\begin{proposition}[Maximum principle] 
\label{propprincmax}
The finite volume method satisfies the following inequalities (in the variable $\u$):
 \be
 \label{princmax}
\min\Big(\uk,~\min_{e\in\partial K}\uke\Big) \le \uuk \le \max\Big(\uk,\max_{e\in\partial K}\uke\Big)
\ee
in each polyhedron $K$ in $\Th$ and at all time level $t^n$.
 \end{proposition}

Since $v_0\in W^{2,\infty}$ immediately implies a sup--norm
estimate for $\vh$ given by (\ref{defve}), we easily deduce, from the maximum principle (\ref{princmax}), an
additional uniform sup-norm estimate but for $w_h=\cz(\uh,\vh)$ arguing about the regularity properties
(\ref{smoothC}) of $\cz$:
\be
\label{supnormw}
\vert\vert \w_h\vert\vert_{L^\infty(\R^+\times\R^d)} \le {\cal O}(1).
\ee

Besides the monotonicity assumption (\ref{gekmonot}) met by the numerical flux functions, we stress that
the preservation of conservativity (\ref{consrec}) in the subcell reconstruction procedure plays a central
role in the validity of the reported maximum principle, as highlighted in the proof. The latter
will be carried out using a recursion procedure based on subsequent partitions of the set of edges $e$ in
$K$. To fix the notation and up to some relabeling, $\{e_1, \ldots, e_{J_K}\}$ represents the full set of
edges $e\in\partial K$ so that here the index $J_K$ is given by $\#\{e,e\in\partial K\}$. Subsets of the
form $\{e_1,...,e_J\}$, with increasing index $J\in \{1,...,J_K\}$, will be of concern as follows. Being
given $J$ with $1\le J\le K$, let us attach to the subset $\{e_1,...,e_J\}$ the solution
$\u^{n+1-}_{K,\{e_1,...,e_J\}}$ of the following nonlinear equation:
\be
\label{urecursive}
\displaystyle \sum_{1\le j\le J} \alpha_{K,e_{j}} \cz(\u^{n+1-}_{K,\{e_1,...,e_J\}},v_{e_j}) = \sum_{1\le
j\le J} \alpha_{K,e_{j}}w^{n+1-}_{K,{e_j}},
\ee
where the subcell states $w^{n+1-}_{K,{e_j}}$ are defined in (\ref{defwwkve}), Lemma
\ref{lemwwkve}. Again, the constitutive assumptions (\ref{smoothC})--(\ref{monotC0}) ensure existence and
uniqueness of a solution to (\ref{urecursive}).

Arguing about the conservation property (\ref{consrec}) satisfied at the subcell reconstruction step, it is
worth observing that $\u^{n+1-}_{K,\{e_1,...,e_{J_K}\}}$ can be identified with the final state $\uuk$ at
time $t^{n+1}$ in the finite volume approximation (\ref{defwkve})--(\ref{defwwk}). Therefore, the
recursion under consideration naturally ends up as soon as the index $J$ reaches the value $J_K$. In order
to initiate the recursion and propagate it, we need the following statement concerned with the values
$u^{n+1-}_{K,\{e_J\}}$, $1\le J\le J_K$, solutions of
$
\cz(u^{n+1-}_{K,\{e_J\}},v_{e_J}) = w^{n+1-}_{K,e_J}.
$

\begin{lemma}[Local maximum principle] 
\label{mpstart}
The maximum principle holds true at any edge $e_J$ in $\partial K$:
$$
\min(\uk,\u^n_{K_{e_J}})\le u^{n+1-}_{K,\{e_J\}} \le  \max(\uk,\u^n_{K_{e_J}}), \quad 1\le J \le J_K.
$$
\end{lemma} 

Then the maximum principle ``propagates'' to sets $\{e_1,...,e_J\}$, as follows. 

\begin{lemma}
\label{mppropagate}
The solution $u^{n+1-}_{K,\{e_1,...,e_{J}\}}$ to (\ref{urecursive}) with $J\in\{1,...,J_K\}$, obeys the
following maximum principle:
$$
\min\Big(\uk,\min_{1\le j \le J}(u^n_{K_{e_j}})\Big) \le u^{n+1-}_{K,\{e_1,...,e_{J}\}} \le
\max\Big(\uk,\max_{1\le j \le J}(u^n_{K_{e_j}})\Big).
$$
\end{lemma}

The proposed lower and upper bounds for $\u^{n+1-}_{K,\{e_1,...,e_{J_K}\}}$, {\it i.e.} the estimate in the lemma 
 with $J=J_K$, just reads the expected local maximum principle (\ref{princmax}) for
$\uuk$, since again $\uuk$ coincides with $\u^{n+1-}_{K,\{e_1,...,e_{J_K}\}}$ by construction.

\begin{proof}[Proof of Lemma \ref{mpstart}] 
To alleviate the notation we skip the index $J$ and first point out an estimate valid under
the CFL restriction (\ref{CFL}) for any edge $e$ in $\partial K$:
\be
\label{mpmonotony}
\min(\wkve,\wkeve) \le \wwkve \le \max(\wkve,\wkeve)
\ee
as a well-known consequence of the monotonicity assumptions (\ref{gekmonot}) satisfied by the numerical
flux function $\gek(.,.;\ve)$. We then recall that the subcell reconstruction step (\ref{defwkve}) builds
$\wkve = \cz(\uk,\ve)$ while the identity $w^{n+1-}_{K,e} = \cz(u^{n+1-}_{K,\{e\}},v_{e})$ holds from our
definition. We can thus recast (\ref{mpmonotony}) as:
$
\min(\cz(\uk,\ve),\cz(\uke,\ve)) \le \cz(u^{n+1-}_{K,\{e\}},\ve) \le \max(\cz(\uk,\ve),\cz(\uke,\ve)),
$
from which we immediately deduce the desired estimate, namely
$$
\min(\uk,\uke) \le u^{n+1-}_{K,\{e\}} \le \max(\uk,\uke), \quad e\in\partial K
$$
since the function $\cz$ is by assumption (\ref{monotC0}) strictly increasing in its first argument.
\end{proof}

\begin{proof}[Proof of Lemma \ref{mppropagate}]
The desired lower-upper bounds with $J=1$ are stated in Lemma \ref{mpstart}. Then, assuming the validity of the maximum principle at rank $J$, $1 \le J < J_K$, this one is proved to hold at the rank $(J+1)$ starting from (\ref{urecursive}):
$$
\left.
\begin{aligned}
& \displaystyle \sum_{1\le j\le (J+1)} \alpha_{K,e_{j}} \cz(u^{n+1-}_{K,\{e_1,...,e_{(J+1)}\}},v_{e_j}) 
\\
& =
\displaystyle \sum_{1\le j\le J} \alpha_{K,e_{j}}w^{n+1-}_{K,e_{j}} + \alpha_{K,e_{(J+1)}}
w^{n+1-}_{K,e_{(J+1)}},\\
&=  \displaystyle \sum_{1\le j\le J} \alpha_{K,e_{j}} \cz(u^{n+1-}_{K,\{e_1,...,e_{J}\}},v_{e_j}) 
 + ~\alpha_{K,e_{(J+1)}} \cz(u^{n+1-}_{K,e_{(J+1)}},v_{e_{(J+1)}}).
 \end{aligned}
 \right.
$$
We recast the above identity as follows:
$$
\left.
\begin{array}{llll}
& \displaystyle \sum_{1\le j\le J} \alpha_{K,e_{j}} \cz(u^{n+1-}_{K,\{e_1,...,e_{(J+1)}\}},v_{e_j}) -
\sum_{1\le j\le J} \alpha_{K,e_{j}}
\cz(u^{n+1-}_{K,\{e_1,...,e_{J}\}},v_{e_j})  
\\
&= - ~\alpha_{K,e_{(J+1)}}\Big( \cz(u^{n+1-}_{K,\{e_1,...,e_{(J+1)}\}},v_{e_{(J+1)}})-
\cz(u^{n+1-}_{K,e_{(J+1)}},v_{e_{(J+1)}})\Big).
 \end{array}
 \right.
$$
To condense the notation, we introduce the two functions $\u\mapsto \Psi_J(u) = \displaystyle
\sum_{1\le j\le J} \alpha_{K,e_{j}} \cz(u,v_{e_j})$ and $\u\mapsto \psi_{(J+1)}(u) =
\alpha_{K,{e_{(J+1)}}} \cz(u,v_{e_{(J+1)}})$ so as to deduce:
$$
\left.
\begin{aligned}
\Big(\Psi_J(u^{n+1-}_{K,\{e_1,...,e_{(J+1)}\}}) & -\Psi_J(u^{n+1-}_{K,\{e_1,...,e_{J}\}})\Big)
\Big(\psi_{(J+1)}(u^{n+1-}_{K,\{e_1,...,e_{(J+1)}\}})- \psi_{(J+1)}(u^{n+1-}_{K,e_{(J+1)}})\Big) \le 0,
\end{aligned}
\right.
$$
since by assumption (\ref{defake}) $\alpha_{K,e_{(J+1)}} > 0$. But the monotonicity hypothesis
(\ref{monotC0}) on $\cz$ together with again assumption (\ref{defake}) imply that both functions $u
\mapsto \Psi_J(u)$ and $u \mapsto \psi_{(J+1)}(u)$ strictly increase with $\u$ so that the above inequality 
yields 
$$
\min(u^{n+1-}_{K,\{e_1,...,e_{J}\}},u^{n+1-}_{K,e_{(J+1)}}) \le u^{n+1-}_{K,\{e_1,...,e_{(J+1)}\}} \le 
\max(u^{n+1-}_{K,\{e_1,...,e_{J}\}},u^{n+1-}_{K,e_{(J+1)}}). 
$$
Lemma \ref{mpstart} implies 
$
 \min(\uk,\u^n_{K_{e_{(J+1)}}})\le u^{n+1-}_{K,e_{(J+1)}} \le  \max(\uk,\u^n_{K_{e_{(J+1)}}})
$, and the proof is completed. 
 \end{proof}


\section{Entropy inequalities}

\subsection{Preliminaries}

Proposition \ref{propprincmax} asserts sup--norm boundedness for the sequence $\uh$ which in the
absence of an {\it a priori} strong compactness argument, leads us to study the structure of the Young
measure $\mu$ associated with $\{\uh\}_{h>0}$. Recall that such a Young measure represents all the
composite weak-star limits $a(\uh)$ of $\uh$ with continuous functions $a\in{\cal C}^0(\R)$, namely for all 
continuous functions in a single variable
$$
a(\uh)\quad \longrightharpoonup \quad <\mu,a> = \int_\R a(\lambda) d\mu(\lambda),  
$$
weakly-star in $L^\infty$. 
We propose to establish that the measure $\mu$ under consideration reduces to a Dirac measure, and hence
to prove the strong convergence of $\uh$, invoking DiPerna's uniqueness
theorem~\cite{DiPerna85} for entropy measure--valued solutions.

In this section we derive the required discrete entropy inequalities together with the {\it a priori}
estimates that are needed to handle the passage to the limit in the sense of measure valued solutions. In
this respect, the main issue is to assess the relevance of the Young measure $\mu$ in such a limit. Indeed, 
discrete entropy inequalities generically involve numerical flux
functions, that are continuous functions but of (at least) {\it two} arguments: the sequence $\uh(.)$
itself and its shift $\Delta_h\uh = \uh(.+h)$. Nonlinear superposition of possible discrete
oscillations in $\uh$ and its shift $\Delta_h \uh$ may prevent the usual Young measure $\mu$ to represent
the composite weak-star limit of $G(\uh, \Delta_h \uh)$. Counterexamples have been constructed in 
Coquel and LeFloch \cite{CoquelLeFloch93}. Some weak control over possible discrete oscillations is
therefore mandatory in order to justify the applicability of $\mu$ in the limiting form of discrete
entropy inequalities.

The requisite weak estimate corresponds to some estimate of the discrete entropy dissipation
rate in the finite volume approximation. The derivation of several specific estimates with distinctive
features have been the matter of a large literature following Coquel and
LeFloch \cite{CoquelLeFloch90}. (The reader is referred to the introduction where several subsequent
contributions were quoted.) The estimates we derive now generalize the ones in Cockburn, Coquel, and LeFloch \cite{CockburnCoquelLeFloch95}. The entropy dissipation estimate does not allow actually to pass to the weak limit in arbitrary numerical
entropy--flux functions, but nevertheless turns out to be sufficient in order to handle discrete entropy inequalities. The main interest in such an estimate stems from the simplicity of its derivation.

\subsection{Discrete entropy estimates}

We first focus on the derivation of the discrete entropy inequalities and then the required weak estimate. The passage to the limit in the discrete
inequalities is the subject of the following section. After Crandall and Majda \cite{CrandallMajda80}, assumptions (\ref{gekconsi})--(\ref{gekmonot}) on the
numerical flux functions $\gek$ are known to yield a full set of discrete entropy inequalities for
scalar conservation laws. Here and in the light of Section \ref{formulation}, the scalar
conservation laws of concern have to be found locally at each edge $e$ in $\Th$, and take the generic
form
\be
\label{slceloc}
\dt \w+  \nabla \cdot \flux(\w,\v) = 0,
\ee
for a given $\v \in \R$.
Associated entropy pairs were defined earlier in (\ref{entropyflux}--(\ref{entropyfluxw}).
The inequalities stated below are naturally built from the subcell states $\wwkve$ (\ref{defwwkve})
of Lemma \ref{lemwwkve} and in this regard may be understood as subcell entropy inequalities.

\begin{lemma}[Entropy inequalities per cell] 
Let $(\U,\F) : \R \to \R\times \R^d$ be any convex entropy pair for the scalar conservation law
(\ref{slceloc}), where $e$ denotes any edge in $\partial K$ for an arbitrarily $K$ in $\Th$. Then
there exists a numerical entropy flux function $\Gek : \R^2\to \R$ that satisfies the consistency
property
\be
\label{egekconsi}
\Gek(\w,\w;\ve) = \F(\w,\ve)\cdot\nuek,
\ee
the conservation property
\be
\label{egekconse}
\Gek(\w,\w_e;\ve) = - G_{e,K_e}(\w_e,\w;\ve)
\ee
for all reals $\w$ and $\w_e$, so that the following discrete entropy inequality holds 
\be
\label{subcellen}
\U(\wwkve) - \U(\wkve) + \frac{1}{\ake}\frac{\tau \vole}{\volk}\Big(\Gek(\wkve,\wkeve;\ve) -
\F(\wkve,\ve)\cdot \nuek\Big) \le 0.
\ee
\end{lemma}

We refer the reader to \cite{CrandallMajda80} for a proof of this classical result. As
already claimed, the weak estimate will not allow to pass weakly to the limit in arbitrary
numerical entropy flux-functions. We thus propose to merge inequalities (\ref{subcellen}) in such a way
that solely exact entropy flux--functions $\F(\w,\ve)\cdot \nuek$ enter the  weak form.  

\begin{lemma}
\label{lemdwis}
Let $\phi$ be any non--negative test function in ${\cal D}(\R^*_+\times \R^d)$. Define for any 
edge $e$ in $\Th$, the average
\be
\label{phie}
\phi^n_e = \frac{1}{\tau\vole}\int_{t^n}^{t^{n+1}}\int_e \phi(x,t)dxdt.
\ee
Then, the following discrete weak inequality holds 
\be
\label{dwis}
\begin{aligned}
\sumkT \sumek \ake  \Big(\U(\wwkve)& -~ \U(\wkve)\Big) \phi^n_e \volk 
 -  ~\tau \sumkT \sumek  \F(\wkve,\ve)\cdot \nuek \phi^n_e \vole \le 0. 
\end{aligned}
\ee
\end{lemma}

The proof is postponed to the end of this section. We shall easily deduce from the discrete
inequality (\ref{dwis}) the following continuous (in space) inequality.

\begin{proposition}
\label{propdwit}
The finite volume approximation (\ref{defwkve})--(\ref{defwwk}) obeys at each time level $t^n$ the
following (discrete in time) entropy inequality
\be
\label{dwit}
\left.
\begin{array}{lll}
\displaystyle
&\sumkT \sumek \ake  \Big(\U(\wwkve) - \U(\wkve)\Big) \phi^n_e \volk \\
& - \displaystyle \iint_{]{t^n},{t^{n+1}}[\times \R^d}\Q(\u^n_h,v(x)) \cdot \nabla \phi(x,t) +\phi(x,t)
\partial_v\Q (\u^n_h,v(x)) : \nabla v(x) dx dt \\
& \le ~{\cal O}(h) \tau\ \vert\vert \phi\vert\vert_{W^{1,\infty}(]t^n,t^{n+1}[\times
\R^d)}\vert \supp(\phi)\vert.
\end{array}
\right.
\ee
\end{proposition}

The proof, at the end of this section, essentially makes use of the uniform
sup--norm estimate (\ref{princmax}) for the sequence $\uh$ together with the regularity assumption
$v_0\in W^{2,\infty}$.  

Clearly, the Young measure $\mu$ can tackle the weak limit of the space derivatives involved in inequality
(\ref{dwit}) extended to any time interval $(0,T)$, $T>0$. Such a claim then naturally rises the question
of passing to the weak limit in the discrete time derivative. The latter is conveniently
decomposed as 
\be
\label{tddecomp}
\begin{aligned}
& \sumkT \sumek \ake  \Big(\U(\wwkve) - \U(\wkve)\Big) \phi^n_e \volk  \\
&= \sumkT\Big(\U(\wwk) - \U(\wk)\Big) \phi^n_K \volk
-  \sumkT\sumek \ake  \Big(\U(\wwk)-\U(\wwkve)\Big) \phi^n_e \volk \\
&\quad -  \sumkT\sumek \ake  \Big(\U(\wkve)-\U(\wk) \Big) \phi^n_e \volk,
\end{aligned}
\ee
where
\be
\label{phik}
 \phi^n_K = \sumek \ake \phi^n_e.
\ee
The last two error terms entering the righ--hand side of (\ref{tddecomp}) are devoted to sum up
\be
\label{errordei}
\begin{aligned}
&\sumnn \sumkT\sumek \ake  \Big(\U(\wwk)-\U(\wwkve)\Big) \phi^n_e \volk \\
&+ \sumnn \sumkT\sumek \ake  \Big(\U(\wkve)-\U(\wk) \Big) \phi^n_e \volk,
\end{aligned}
\ee
with other error terms in the right--hand side of the discrete entropy inequalities (\ref{dwit}). The
former must therefore be proved to go to zero with $h$.

\begin{lemma} 
\label{lemerrort}
For any polyhedron $K$ in $\Th$, one has 
\be
\label{estiwkwkve}
\sumek \ake \Big(\U(\wkve)-\U(\wk)\Big) \phi^n_e \le {\cal O}(h^2) \vert \vert \phi
\vert\vert_{W^{1,\infty}(]t^n,t^{n+1}[\times K)},
\ee
while
\be
\label{estiwwkwwkve}
\begin{aligned}
& \sumek \ake \Big(\U(\wwk)
-\U(\wwkve)\Big) \phi^n_e 
\\
& \le - \sigma_\U \Big(\sumek \ake \vert \wwk -
\wwkve\vert^2\Big) \phi^n_K 
\\
& \quad + {\cal O}(h) \Big(\sumek \ake \vert \wwk - \wwkve\vert\Big) \vert \vert \nabla \phi
\vert\vert_{L^{\infty}(]t^n,t^{n+1}[\times K)}, 
  \end{aligned}
 \ee
where $\sigma_\U$ denotes some convexity-like modulus of $\U$: $\U''(\u) \ge \sigma_\U > 0,~\hbox{for all
} \u \in (m,M)$ where the bounds $m,M$ were introduced in (\ref{CFL}) in agreement with the maximum
principle (\ref{princmax}).
\end{lemma}

\begin{proof}[Proof of Lemma \ref{lemdwis}.]
Let $e$ be any edge in $\Th$ and $K,K_e$ the associated pair of adjacent polyhedra.
Multiplying the subcell entropy inequality (\ref{subcellen}) valid for $K$ by $\ake \volk$ and the
companion inequality for $K_e$ by $\ake \vert K_e\vert$, we get 
$$
\begin{aligned}
 \ake \volk \Big(\U(\wwkve) &- \U(\wkve)\Big)  + ~ \akee \volke \Big(\U(\wwkeve) - \U(\wkeve)\Big) \\
& -  \tau \Big(\F(\wkve,\ve)\cdot \nuek + \F(\wkeve,\ve) \cdot \nueke  \Big)\vole \le 0, 
\end{aligned}
$$
thanks to the conservation property (\ref{egekconse}) satisfied by the numerical entropy flux--functions.
Multiplying the above inequality by the discrete test function $\phi^n_e$ (\ref{phie}), then summing over
the edges $e$ in $\partial K$ and the polyhedra $K$ in $\Th$ yields
$$
\begin{aligned}
&\sumkT \sumek \ake  \Big(\U(\wwkve) - \U(\wkve)\Big) \phi^n_e \volk  \\ 
& +~\sumkT \sumek \akee  \Big(\U(\wwkeve) - \U(\wkeve)\Big) \phi^n_e \volke \\
& -\tau \sumkT \sumek \Big(\F(\wkve,\ve)\cdot \nuek + \F(\wkeve,\ve) \cdot \nueke \Big)\phi^n_e \vole \le
0.
\end{aligned}
$$
To conclude the proof, we notice the following two identities
$$
\begin{aligned}
& \sumkT \sumek \ake \Big(\U(\wwkve) - \U(\wkve)\Big) \phi^n_e \volk
\\ 
& = \sumkT \sumek \akee
\Big(\U(\wwkeve) - \U(\wkeve)\Big) \phi^n_e \volke,
\end{aligned}
$$
and
$$
\sumkT \sumek \F(\wkeve,\ve) \cdot \nueke \phi^n_e \vole = \sumkT \sumek \F(\wkve,\ve) \cdot \nuek
\phi^n_e \vole.
$$
\end{proof}

\begin{proof}[Proof of Proposition \ref{propdwit}.]
We begin with the discrete inequality (\ref{dwis}) of Lemma \ref{lemdwis} and specifically considerb 
the flux balance 
$
\sumkT \sumek  \F(\wkve,\ve)\cdot \nuek \phi^n_e \vole.
$
Our purpose is to shift the mathematical expressions under consideration from the $\w$ to the $\u$
variable. Hence let us write $\F(\wkve,\ve) = \F(\w(\uk,\ve),\ve) = \Q(\uk,\ve)$ with $\Q(\u,\v) $ the
exact entropy flux introduced in (\ref{entropyflux}), which we repeat component-wise as 
$
\Q_i(\u,\v) = \int^\u \U'(\cz(\theta,\v))\partial_\theta \ci(\theta,\v) d\theta, \quad 1\le i \le d.
$
We then recast the flux balance as
\be
\label{recastfb}
\begin{aligned}
& \sumek  \F(\wkve,\ve) \cdot \nuek  \phi^n_e \vole 
\\
& =  \Q(\uk,\vk) \cdot \sumek\phi^n_e \vole \nuek  + \sumek \Big(\Q(\uk,\ve)- \Q(\uk,\vk)  \Big)\cdot \nuek \phi^n_e \vole,
\end{aligned}
\ee
where the average of the states $\ve$ is defined by 
$
\vk = \sumek \ake \ve.
$
In view of a representation formula for $\nabla \phi$ (similar to the one in Remark~\ref{gradve} derived for
$\nabla v$), the average form (\ref{phie}) for $\phi^n_e$ yields 
\be
\sumek\phi^n_e \vole \nuek = \frac{1}{\tau} \int_{t^n}^{t^{n+1}}\!\!\Big(\sumek \int_e \phi(x,t) \nuek dx
\Big)dt =\frac{1}{\tau} \int_{t^n}^{t^{n+1}}\!\!\! \int_K \nabla \phi(x,t) dx dt,
\ee
so that, from (\ref{recastfb}),
\be
\label{tempfb}
\begin{aligned}
& \sumek \F(\wkve,\ve) \cdot \nuek \phi^n_e \vole
\\
& = \displaystyle\frac{1}{\tau} \int_{t^n}^{t^{n+1}}\int_K
\Q(\uk,\vk) \cdot \nabla \phi(x,t) dx dt  + \sumek \Big(\Q(\uk,\ve)- \Q(\uk,\vk)  \Big)\cdot \nuek \phi^n_e \vole.
\end{aligned}
\ee
The treatment of the last remaining discrete term relies on the following identity:
$$
\left.
\begin{array}{lll}
\displaystyle
\Q(\uk,\ve)- \Q(\uk,\vk) =  \int_0^1 \partial_v \Q(\uk,\vk + s(\ve-\vk)) ds~ (\ve -\vk)
\end{array}
\right.
$$
which leads us to rewrite (\ref{tempfb}): 
\be
\label{interfb}
\left.
\begin{array}{lll}
& \sumek \F(\wkve,\ve) \cdot \nuek \phi^n_e \vole - \displaystyle\frac{1}{\tau} \int_{t^n}^{t^{n+1}}\int_K
\Q(\uk,\vk) \cdot \nabla \phi(x,t) dx dt
\\
& =  \partial_v\Q(\uk,\vk): \Big(\sumek \phi^n_e (\ve-\vk)\otimes\nuek\vole\Big) 
\\
& \quad + ~ \sumek \phi^n_e \Big(\int_0^1 \big(\partial_v \Q(\uk,\vk + s(\ve-\vk))-\partial_v \Q(\uk,\vk)\big)
ds \Big) : \Big((\ve-\vk)\otimes\nuek\Big) \vole.
\end{array}
\right.
\ee
The matrix $(\ve-\vk)\otimes\nuek\vole$ with size $L\times d$ appears as a discrete representation for the
continuous function $\nabla v$.
The first term in the above right--hand side is rewritten as:
\be
\label{ftrhsinterfb}
\begin{aligned}
& \partial_v\Q(\uk,\vk): \Big(\sumek  \phi^n_e (\ve-\vk)\otimes\nuek \vole \Big)   
\\
&=  \phi^n_K ~\partial_v\Q(\uk,\vk): \Big(\sumek (\ve-\vk)\otimes\nuek \vole\Big) \\ 
& \quad + \Big(\sumek ( \phi^n_e - \phi^n_K) \partial_v\Q(\uk,\vk): \big((\ve-\vk)\otimes\nuek\big) \vole \Big),
\end{aligned}
\ee
where the discrete flux function $\phi^n_K$ is obtained by averaging:
$
\phi^n_K = \sumek \ake \phi^n_e.
$
On one hand, owing to the identity $\sumek (\ve-\vk)\otimes\nuek\vole  = \sumek\ve\otimes\nuek\vole$
we get
\be
\label{contrepre}
\begin{aligned}
& \partial_v\Q(\uk,\vk):\Big(\sumek (\ve-\vk)\otimes\nuek \vole\Big) 
= \partial_v\Q(\uk,\vk): 
 \left(\displaystyle \frac{1}{\tau}\int_{t^n}^{t^{n+1}}\!\!\!\int_K \nabla v(x) dt dx\right),
\end{aligned}
 \ee
again thanks to the representation formula in Remark~\ref{gradve} (for $\nabla v$). On the other hand, the latter 
error term in (\ref{ftrhsinterfb}) is described by 
\be
\label{lasterrorterm}
\begin{aligned}
& \Big\vert \sumek ( \phi^n_e - \phi^n_K)
\partial_v\Q(\uk,\vk): ((\ve-\vk)\otimes\nuek)\vole \Big\vert 
\\
& \le {\cal O}(1) \displaystyle \sup_{e\in\partial K} \vert(\phi^n_e - \phi^n_K)(\ve - \vk)\vert) p_K  
\\
& \le {\cal O}(h_K^2) \displaystyle \vert\vert \nabla \phi \vert \vert_{L^\infty(]t^n,t^{n+1}[\times K)} {
p_K} 
 \le {\cal O}(h_K)~ \displaystyle \vert\vert \nabla \phi \vert\vert_{L^\infty(]t^n,t^{n+1}[\times K)}
\volk.
\end{aligned}
\ee
Here, we have successively used the sup--norm estimate (\ref{princmax}) satisfied by $\uh$, the
definition of the perimeter $\pk$ of $K$, the estimate
\be
\label{oscilvevk}
\vert \ve-\vk\vert \le \displaystyle \sum_{e'\in \partial K}\ake \vert \ve-v_{e'}\vert \le {\cal O}(h_K)
\ee 
from the definition of $\vk$ and the regularity property $v_0\in W^{2,\infty}$,
a similar estimate $\vert \phi^n_e-\phi^n_k\vert \le {\cal O}(h_K)$ and finally the non degeneracy
assumption (\ref{nondegenere}) on the triangulation $\Th$. Involving
\eqref{contrepre}-\eqref{lasterrorterm}, the identity \eqref{ftrhsinterfb} yields the following estimate
\be
\label{ftrhfb}
\begin{aligned}
& \Big\vert \partial_v\Q(\uk,\vk) : \Big(\sumek \phi^n_e (\ve-\vk)\otimes\nuek\vole \Big)
- 
\displaystyle \frac{1}{\tau}\int_{t^n}^{t^{n+1}}\int_K \phi^n_K \partial_v\Q(\uk,\vk) :
\nabla v(x) dt dx
\Big\vert \\
& \le {\cal O}(h_K) ~ \displaystyle \vert\vert \nabla \phi \vert\vert_{L^\infty(]t^n,t^{n+1}[\times K)}
\volk.
\end{aligned}
\ee

For the final error term in the flux balance (\ref{interfb}), we have the following
bounds:
$$
\begin{aligned}
&\Big\vert \sumek \phi^n_e \displaystyle \int_0^1 (\partial_v \Q(\uk,\vk + s(\ve-\vk))-\partial_v
\Q(\uk,\vk)) ds : ((\ve-\vk)\otimes\nuek) \vole\Big\vert \\
& \le {\cal O}(1) \displaystyle \sup_{e\in \partial K} \vert \ve-\vk\vert^2 \Big( p_K \vert\vert
\phi\vert\vert_{L^\infty(]t^n,t^{n+1}[\times K)}\Big)\\
& \le {\cal O}(h_K)\vert\vert \phi\vert\vert_{L^\infty(]t^n,t^{n+1}[\times K)} \volk,
\end{aligned}
$$
where we have used the regularity of the entropy flux $\Q$, the sup--norm
estimate (\ref{princmax}), the estimate (\ref{oscilvevk}) satisfied by
$\vert\ve-\vk\vert$, and the non-degeneracy assumption (\ref{nondegenere}) on the triangulation
$\Th$.

To summarize, we have obtained the estimate for the flux balance on a single cell:
\be
\label{lastfb}
\begin{aligned}
 \Big\vert 
& \displaystyle\frac{1}{\tau} \int_{t^n}^{t^{n+1}}\!\!\int_K \Big(\Q(\uk,\vk) \cdot \nabla \phi(x,t) 
 +  \phi^n_K \partial_v\Q(\uk,\vk) :  \nabla v(x) dt dx \Big)\\
& \hskip2.cm -  \sumek  \F(\wkve,\ve) \cdot \nuek  \phi^n_e \vole \Big\vert  \\
&  \le {\cal O}(h) \vert\vert \phi\vert\vert_{W^{1,\infty}(]t^n,t^{n+1}[\times K)} \volk.
\end{aligned}
 \ee
From the discrete weak entropy inequality \eqref{dwis} we recall that 
$$
 \sumek \ake  \Big(\U(\wwkve) - \U(\wkve)\Big) \phi^n_e \volk -  
\tau \sumkT \sumek  \F(\wkve,\ve)\cdot \nuek \phi^n_e \vole \le 0,
$$
the sum of \eqref{lastfb} over all cells $K$ on the triangulation $\Th$ gives
$$
\begin{aligned}
& \sumkT \sumek \ake  \Big(\U(\wwkve)  - \U(\wkve)\Big) \phi^n_e \volk   \\
& - \displaystyle \int_{t^n}^{t^{n+1}}\Big(\sumkT \int_K \Q(\uk,\vk)\cdot\nabla \phi+ \phi
\partial_v\Q(\uk,\vk):\nabla v dx\Big)dt \\
& \le {\cal O}(h) \displaystyle \tau \sumkT \vert\vert \phi\vert\vert_{W^{1,\infty}(]t^n,t^{n+1}[\times
K)} \volk
\le {\cal O}(h) \tau \vert\vert \phi\vert\vert_{W^{1,\infty}(]t^n,t^{n+1}[\times \R^d)}\vert
\supp(\phi)\vert.
\end{aligned}
$$
\end{proof}

\begin{proof}[Proof of Lemma \ref{lemerrort}]
We first establish the estimate (\ref{estiwkwkve}) and consider the following decomposition
involving again the $\{\ake\}_{\{e,e\in\partial K\}}$-average $\phi^n_K$ of the $\phi^n_e$ (\ref{phik}):
$$
\begin{aligned}
& \sumek \ake \Big(\U(\wkve)-\U(\wk)\Big) \phi^n_e
\\
 &= \sumek \ake \Big(\U(\wkve)- \U(\wk)\Big)
(\phi^n_e-\phi^n_K) + \phi_K^n~\Big(\sumek \ake \U(\wkve)-\U(\wk)\Big),
\end{aligned}
$$
from which we deduce the following bound:
\be
\label{intwkwkve}
\begin{aligned}
& \sumek \ake \Big(\U(\wkve)-\U(\wk)\Big) \phi^n_e \\
& \le {\cal O}(h_K) \vert\vert \nabla\phi\vert\vert_{L^{\infty}(]t^n,t^{n+1}[\times K)} 
\displaystyle \sup_{e\in\partial K}\vert \wkve-\wk\vert \\
& \quad + ~ {\cal O}(1) ~ \vert\vert \phi\vert\vert_{L^{\infty}(]t^n,t^{n+1}[\times K)} \Big(\sumek \ake
\U(\wkve)-\U(\wk)\Big),
\end{aligned}
\ee
in view of the sup-norm estimate (\ref{princmax}) satisfied by $\uh$, the estimate
$\vert\phi^n_e-\phi^n_K\vert\le {\cal O}(h_K)$ and the convexity of the entropy $\U(\w)$.
The first error term in (\ref{intwkwkve}) is given the following bound:
\be
\label{oscilwkve}
\begin{aligned}
\vert \wkve-\wk \vert &\le \displaystyle \sum_{e'\in\partial K} \alpha_{K,e'}
\big\vert\cz(\uk,v_{e'})-\cz(\uk,\ve)\big\vert \\
&\le {\cal O}(1) \displaystyle \sup_{e'\in \partial K} \vert \v_{e'}-\ve\vert \le {\cal O}(h_K),
\end{aligned}
\ee
while the second one may be handled as follows:
\be
\label{osciluwkve}
\begin{aligned}
& \sumek \ake \U(\wkve)-\U(\wk) = \U'(\wk)\Big(\sumek \ake \wkve - \wk\Big) \\
&+ \sumek \ake \int_0^1 \U''(\wkve + s(\wk-\wkve))ds (\wkve-\wk)^2\\
&\le {\cal O}(1) \displaystyle \sup_{e\in\partial K} \vert\wkve-\wk\vert^2 ~\le ~ {\cal O}(h_K^2),
\end{aligned}
\ee
in view of  (\ref{defwk}) $\wk =\sum_{e\in\partial K} \ake \wkve$ and the estimate
(\ref{oscilwkve}). Gathering bounds (\ref{oscilwkve}) and (\ref{osciluwkve}) yield the expected estimate
(\ref{estiwkwkve}) in Lemma \ref{lemerrort}.

We now derive the companion estimate (\ref{estiwwkwwkve}), by starting from the decomposition 
$$
\begin{aligned}
& \sumek \ake \Big(\U(\wwk)-\U(\wwkve)\Big) \phi^n_e 
\\
& = ~\phi^n_K ~\Big(\U(\wwk)-\sumek \ake
\U(\wwkve)\Big)
+ \sumek \ake \Big(\U(\wwk)-\U(\wwkve)\Big) \Big(\phi^n_e-\phi^n_K\Big),
 \end{aligned}
$$
and observing, on one hand, 
$$
\begin{aligned}
& \Big\vert \sumek \ake \Big( \U(\wwk)-\U(\wwkve)\Big) \Big(\phi^n_e-\phi^n_K\Big)\Big\vert 
\\
& \le {\cal O}(1) \displaystyle\sum_{e\in\partial K} \ake \vert\phi^n_e-\phi^n_K\vert \vert
\wwkve-\wwk\vert \\
& \le {\cal O}(h_K) \Big(\sumek \ake \vert \wwkve-\wwk\vert \Big)\vert\vert \nabla
\phi\vert\vert_{L^\infty(]t^n,t^{n+1}[\times K)} 
 \end{aligned}
$$
and, on the other hand, 
$$
\begin{aligned}
& \sumek \ake \U(\wwkve)-\U(\wwk) 
\\
&= \U'(\wwk)\Big(\sumek \ake \wwkve - \wwk\Big) 
\\
& \quad + \sumek \ake \int_0^1 \U''(\wwkve + s(\wwk-\wwkve))ds (\wwkve-\wwk)^2.
\end{aligned}
$$
Finally, in view of the convex decomposition (\ref{convwwk}) stating $\wwk = \sum_{e\in\partial K} \ake
\wwkve$
$$
\U(\wwk) - \sumek \ake \U(\wwkve) \le  - \sigma_\U \sumek \ake \vert \wwkve-\wwk\vert^2, 
$$
where $\sigma_\U$ denotes the convexity like-modulus of $\U$ introduced in Lemma \ref{lemerrort}. This
concludes the proof.
\end{proof}


\subsection{Entropy dissipation rate and strong convergence}
\label{SectConv}

The proposed estimates obtained in Lemma~\ref{lemerrort} deserve a few comments. Plugging first estimate
(\ref{estiwkwkve}) in (\ref{errordei}) will be easily seen to yield the following upper-bound
$$
\begin{aligned}\sumnn \sumkT\sumek \ake  \Big(\U(\wkve)&-\U(\wk) \Big) \phi^n_e \volk
 \le {\cal O}(h) \vert \vert \phi \vert\vert_{W^{1,\infty}(\R_+\times \R^d)}\vert \supp(\phi)\vert
\end{aligned}
$$
that obviously suffices to conclude. By contrast and turning considering (\ref{estiwwkwwkve}), a crude
upper-bound based on the sup-norm estimate (\ref{supnormw}), say
$$\sumek \ake \Big(\U(\wwk)-\U(\wwkve)\Big) \phi^n_e \le {\cal O}(h) \vert \vert \phi
\vert\vert_{W^{1,\infty}(]t^n,t^{n+1}[\times K)}$$
would result in the useless estimate 
$$
\begin{aligned}\sumnn \sumkT\sumek \ake  \Big(\U(\wwk)-&\U(\wwkve)\Big) \phi^n_e \volk
\le {\cal O}(1) \vert \vert \phi \vert\vert_{W^{1,\infty}(\R_+\times \R^d)}\vert
\supp(\phi)\vert.
\end{aligned}
$$
Proving that the error term of concern in (\ref{errordei}) actually vanishes with $h$ requires therefore
in turn a sharper control in (\ref{estiwwkwwkve}) of the oscillations of the $\wwkve$ around their mean
value $\wwk$. Such a control over these discrete oscillations results from a sharp evaluation of the
discrete entropy rate of dissipation.

\begin{proposition}
\label{propweakesti}
Let $T>0$ be any fixed time and let $N_T\in\mathbb{N}$ be the floor of $T/\tau$ we denote
$[T/\tau]$.
Then, for any (time independent) non negative test function $\psi\in {\cal D}(\R^d)$, the finite
volume approximation (\ref{defwkve})--(\ref{defwwk}) obeys the following estimate on the discrete
oscillations:
\be
\label{weakesti}
\sum_{n=0}^{N_T} \sumkT \sumek \ake \vert \wwk - \wwkve\vert ^2 \psi_K \volk \le {\cal O}(1),
\ee
where $\psi_K$ reads 
$
\psi_K=\sumek \ake \psi_e, \qquad \psi_e = \frac{1}{\vole}\int_e \psi(x)dx.
$
\end{proposition}

Equipped with (\ref{weakesti}) we obtain the following entropy dissipation rate.

\begin{corollary}
\label{corentine}
The sequence $\uh$ satisfy the entropy like inequality
 \be
 \label{cwts}
\displaystyle\iint_{\R^+\times \R^d}\!\!\!\!\! \U(\cz(\uh,\v)) \partial_t \phi(x,t) +
 \Q(\uh,\v) \cdot \nabla \phi +\phi \partial_v\Q(\uh,\v) \!:\! \nabla v dx dt  \ge 
{\cal O}(h^{1/2}),
\ee
for any (smooth) convex entropy pair $(\U,\Q) : \R \to \R\times \R^d$ introduced in \eqref{entropyu}
and \eqref{entropyflux}.
\end{corollary}

Equipped with the above inequality valid for any entropy pair $(\U,\Q)$, we easily deduce that the
Young measure $\mu=\mu_{t,x}$ associated with the sequence $(u_h)_{h>0}$ is an entropy satisfying measure
valued solution. In other words 
the uniformly bounded $L^\infty$ sequence $(u_h)_{h>0}$, as announced at the beginning of this section, it is easy to
check that the inequation \eqref{cwts} becomes as $h$ tends to 0 the following inequation satisfied in the
weak sense:
\be
 \label{mventropy}
\partial_t \langle \mu, \U(\cz(\cdot,\v))\rangle +
\nabla_x \langle \mu, \Q(\cdot,\v) \rangle - \langle \mu , \partial_v \Q(\cdot,\v)\rangle \!:\! \nabla v
\le 0.
\ee
Relying on a direct extension of DiPerna's uniqueness theorem \cite{BoutinCoquelLeFloch09c}, we can 
deduce that the entropy measure--valued solution $\mu_{t,x}$ reduces to a Dirac measure
$\delta_{u(t,x)}$ concentrated on a function $u=u(t,x)$ since the initial data $\mu_0$ coincides with the
Dirac measure $\delta_{u_0}$ (where $u_0$ is the initial data in the Cauchy problem
(\ref{cauchyuv})). Proving that the inital data $u_0$ is correctly handled amounts to show that for every
compact subset ${\cal K}$ of $\R$ we have
\be
\label{mvinitialdata}
\lim_{t\to 0+}\int_0^t \int_{\cal K} \langle \mu_{s,x} , |id-u_0(x)|\rangle\ dxds = 0.
\ee 
The condition
\eqref{mvinitialdata}-\eqref{mventropy} reduces to a Dirac measure concentrated at $u(t,x)$, the Kruzkov
entropy solution of \eqref{cauchyuv}-\eqref{entropyu} with same initial data $u_0$. In other words, for
all time $T>0$ and for all compact $\cal K$ in $\R$, the scheme converges strongly in
$L^p_{loc}((0,T)\times{\cal K})$ to the solution $u$. Theorem~\ref{maintheo} of this paper is
thus now established.

\begin{proof}[Proof of Proposition \ref{propweakesti}.]
We start from the discrete in time weak formulation \eqref{dwit} stated in Proposition~\ref{propdwit}: 
$$
\left.
\begin{array}{lll}
\displaystyle
&\sumkT \sumek \ake  \Big(\U(\wwkve) - \U(\wkve)\Big) \phi^n_e \volk \\
& - \displaystyle \iint_{]{t^n},{t^{n+1}}[\times \R^d}\Q(\u^n_h,v(x)) \cdot \nabla \phi(x,t) +\phi(x,t)
\partial_v\Q (\u^n_h,v(x)) \!:\!\nabla v(x) dx dt 
\\
& \le ~{\cal O}(h) \tau\ \vert\vert \phi\vert\vert_{W^{1,\infty}(]t^n,t^{n+1}[\times
\R^d)}\vert \supp(\phi)\vert,
\end{array}
\right.
$$
in which we plug the decomposition \eqref{tddecomp}-\eqref{phik}. A discrete test function $\psi_K$ given
for any given time-independent test function $\psi\in\calD(\R^d)$ is considered. We then
get
$$
\begin{aligned}
& \sumkT \Big(\U(\wwk)-\U(\wk)\Big)\psi_K \volk 
\\
& - \displaystyle \iint_{]{t^n},{t^{n+1}}[\times \R^d} \!\!\!\!\Q(\u^n_h,v(x)) \cdot \nabla \psi(x)
+\psi(x) \partial_v\Q (\u^n_h,v(x))\!:\!\nabla v(x) dx dt \\
& \le \displaystyle\sumkT \sumek \ake \Big(\U(\wwk)-\U(\wwkve)\Big) \psi_e \volk 
+ \displaystyle \sumkT\sumek  \ake \Big(\U(\wkve)-\U(\wk))\Big) \psi_e \volk  \\
& \quad +~{\cal O}(h) \tau \vert\vert \psi\vert\vert_{W^{1,\infty}(\R^d)} \vert \supp(\psi)\vert.
\end{aligned}
$$

Invoquing estimates \eqref{estiwkwkve}-\eqref{estiwwkwwkve} then yields
$$
\begin{aligned}
& \sumkT \Big(\U(\wwk)-\U(\wwkve)\Big)\psi_K \volk  + \sigma_\U \sumkT \sumek \ake \vert \wwk - \wwkve\vert^2 \psi_K \volk 
\\
& \le  {\cal O}(h) \tau \vert\vert \psi\vert\vert_{W^{1,\infty}(\R^d)}\vert \supp(\psi)\vert
 +~{\cal O}(h) \sumkT \vert\vert \nabla\psi\vert\vert_{L^\infty(K)} \volk
  +~{\cal O}(h^2) \sumkT \vert\vert \psi\vert\vert_{W^{1,\infty}(K)} \volk
\\
& \quad + \displaystyle \iint_{]{t^n},{t^{n+1}}[\times \R^d} \!\!\!\!\Q(\u^n_h,v(x)) \cdot \nabla \psi(x)
+\psi(x) \partial_v\Q (\u^n_h,v(x))\!:\!\nabla v(x) dx dt.
\end{aligned}
$$
Observe that due to the estimate \eqref{princmax}, the last contribution in the above right--hand side can
be given the following crude estimate ${\cal O}(\tau) \vert\vert \psi\vert \vert_{W^{1,\infty}(\R^d)}$.
Henceforth, we deduce that
$$
\begin{aligned}
&
\sumkT \Big(\U(\wwk)-\U(\wk)\Big)\psi_K \volk + \sigma_\U \sumkT \Big(\sumek \ake \vert \wwk -
\wwkve\vert^2 \Big)\psi_K \volk \\
& \le ~{\cal O}(h)\vert\vert \psi\vert \vert_{W^{1,\infty}(\R^d)}.
\end{aligned}
$$
Summing over time indices $n\in[0,N_T]$ with $N_T=[T/\tau]$ for a fixed time $T>0$, we get
$$
 \begin{aligned}
& \int_{\R^d} \U(\wh(x,T))\psi_h(x) dx + \sigma_\U \sum_{n=0}^{N_T}\sumkT \Big(\sumek \ake \vert \wwk -
\wwkve\vert^2 \Big)\psi_K \volk \\
& \le \int_{\R^d} \U(\w_0(x))\psi_h(x) dx + {\cal O}(1)T\vert\vert \psi\vert \vert_{W^{1,\infty}(\R^d)},
 \end{aligned}
$$
which is the required result.
\end{proof}

\begin{proof}[Proof of Corollary \ref{corentine}]
We start from \eqref{dwit}-\eqref{tddecomp}-\eqref{phik} and consider the following discrete in
time weak formulation for the time dependent test function $\phi\in\calD(\R^+_*\times\R^d)$ and its
discrete representation $\phi_K^n$
$$
\begin{aligned}
& \sumkT \Big(\U(\wwk)-\U(\wk)\Big)\phi^n_K \volk \\
& - \displaystyle \iint_{]{t^n},{t^{n+1}}[\times \R^d} \!\!\!\!\Q(\u^n_h,v(x)) \cdot \nabla \phi(x,t)
+\phi(x,t) \partial_v\Q (\u^n_h,v(x)) \!:\! \nabla v(x) dx dt\\
& \le {\cal O}(h) \tau \vert\vert \phi\vert\vert_{W^{1,\infty}(]t^n,t^{n+1}[\times\R^d)} \vert
\supp(\phi)\vert
+ {\cal O}(h^2) \sumkT \vert\vert \phi\vert\vert_{W^{1,\infty}(]t^n,t^{n+1}[\times K)} \volk\\
& +~{\cal O}(h) \sumkT \sumek \ake \vert\wwk-\wwkve\vert\
\vert\vert\nabla\phi\vert\vert_{L^\infty(]t^n,t^{n+1}[\times K)} \volk.
\end{aligned}
$$
where we have used estimates \eqref{estiwkwkve}-\eqref{estiwwkwwkve}. Summing this inequality over time
indices gives
\be
\label{weakentrop}
 \begin{aligned}
  & -\sum_{n\ge 0} \sumkT \U(\wwk)\dfrac{\phi_K^{n+1}-\phi_K^n}{\tau}\tau\volk\\
& -\iint_{\R^+\times\R^d} \Q(\u^n_h,v(x)) \cdot \nabla \phi(x,t) +\phi(x,t) \partial_v\Q (\u^n_h,v(x))
\!:\! \nabla v(x) dx dt\\
  &\quad \leq {\cal O}(h) \vert\vert \phi \vert\vert_{W^{1,\infty}(\R^+\times\R^d)}\\
&\quad + {\cal O}(1) \sum_{n\ge 0} \sumkT \Big(\sumek \ake \vert\wwk-\wwkve\vert\ \chi_\phi
\vert\vert\nabla\phi\vert\vert_{L^\infty(]t^n,t^{n+1}[\times K)} \volk \tau\Big),
 \end{aligned}
\ee
making use of the characteristic function $\chi_\phi$ of $\bigcup_{0<t<T} \supp(\phi(\cdot,t))$, a compact subset of~$\R^d$, 
where $T$ is a finite time such that $\supp(\phi(\cdot,t))=\emptyset$ for $t\ge T$.
Cauchy-Schwarz's inequality then yields the following crude upper bound for the last term: 
$$
 \begin{aligned}
&\sum_{n\ge 0} \sumkT \Big(\sumek \ake \vert\wwk-\wwkve\vert\
\chi_\phi\Big)\vert\vert\nabla\phi\vert\vert_{L^\infty(]t^n,t^{n+1}[\times K)} \volk \tau\\
  &\leq
\Big(\sum_{n\ge 0} \sumkT \big(\sumek  \ake \vert\wwk-\wwkve\vert\ \chi_\phi\big)^2 \volk\tau\Big)^{1/2}
\Big(\sum_{n\ge 0} \sumkT \vert\vert\nabla\phi\vert\vert^2_{L^\infty(]t^n,t^{n+1}[\times
K)}\volk\tau\Big)^{1/2}\\
&\leq {\cal O}(1) \Big(\sum_{n\ge 0} \sumkT \big(\sumek \ake \vert\wwk-\wwkve\vert^2\big) \chi_\phi
\volk\tau\Big)^{1/2}
 \end{aligned}
$$
as a consequence of the convexity property of the $\ake-$average. The estimate \eqref{weakesti} then
yields with $\psi=\chi_\phi$
$$
 \begin{aligned}
\sum_{n\ge 0} \sumkT \Big(\sumek \ake \vert\wwk-\wwkve\vert\
\chi_\phi\Big)\vert\vert\nabla\phi\vert\vert_{L^\infty(]t^n,t^{n+1}[\times K)} \volk \tau \leq {\cal
O}(h^{1/2}).
 \end{aligned}
$$
Then routine arguments give the conclusion from \eqref{weakentrop}.
\end{proof}


\section{Numerical experiments}

\subsection{A two domain coupling problem}

In this first test, we consider an heterogeneous medium which occupies the spatial domain $[-1,1]^2$ and is 
constituted by an annular inclusion $\calD_1$ centered at the origin $(0,0)$ with external
radius~$\sqrt{0.2}$ and with internal radius~$\sqrt{0.1}$, and by its complement set $\calD_0$. In these
two domains, the following respective flux--functions are considered in term of the scalar unknown $w=w(t,x)$:
\begin{equation*}
f_0(w)={w^2 \over 2} \begin{pmatrix}1\cr 1\end{pmatrix},\quad \qquad
f_1(w)= {(w-0.9)^2 \over 2} \begin{pmatrix}1\cr
1\end{pmatrix}.
\end{equation*}
The regularized color function $v$ plotted in Figure~\ref{fig-multiD-v} provides us with a regularized version of the characteristic function of the domain $\calD_1$.
The coupling condition between $\calD_0$ and $\calD_1$ takes here the form
\begin{equation*}
2 \, w_-(t,x) = w_+(t,x),\quad x\in\partial\calD_1,
\end{equation*}
where $w_{\pm}(t,x)=\lim_{\theta\to 0+} w(t,x\pm\theta \nu_x)$ and $\nu_x$ the exterior unit normal at $x\in\partial\calD_1$.

The initial data plotted in Figure~\ref{fig-multiD-w0} is piecewise constant: 
\begin{equation*}
 w_0(x,y)=
\begin{cases}
1,&x<-0.8,\\
0,&x\geq -0.8.
\end{cases}
\end{equation*}
The computations are performed on a Cartesian grid with~$100\times100$ meshes, and the CFL number is chosen to be~$0.5$.

In an homogeneous domain with the sole flux $f_0$, such an initial data would develop a shock front moving
with the speed vector $0.5(1,1)^T$. In the present heterogeneous domain, this shock front has the
same behavior only until it reaches the interface between both domains (see Figures~\ref{fig-multiD-wA}). The
coupling condition at this interface is such that the value $w=2$ arises then inside the domain
$\calD_+$. In this second domain, where the flux under consideration is $f_1$, we observe then a
(curved) shock wave connecting the states $w=2$ and $w=0$ and moving at the fixed speed given by the
Rankine--Hugoniot relation, that is, $0.605(1,1)^T$ (see Figures~\ref{fig-multiD-wC} and~\ref{fig-multiD-wE}).
Finally, the shock front goes outside the whole domain $[-1,1]^2$ (see Figure~\ref{fig-multiD-wH}). 
In Figures~\ref{fig-multiD-uA}, \ref{fig-multiD-uC}, \ref{fig-multiD-uE}, and \ref{fig-multiD-uH}), we
plot the $u$--variable, which is found to remain constant at each interface, as expected by the theory.

\begin{figure}[!ht]
\centering
{
\hfill%
\subfloat[\label{fig-multiD-w0}Initial data
$w_0$]{\includegraphics[angle=0,width=.45\linewidth,keepaspectratio,clip,trim=150 30 100
120]{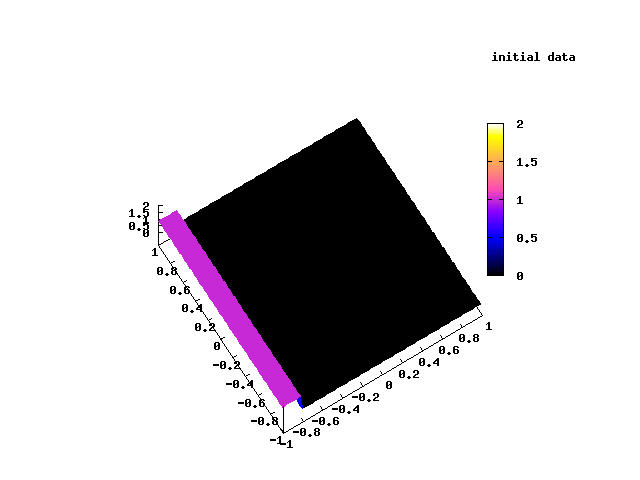}}%
\hfill
\subfloat[\label{fig-multiD-v}Color function
$v$]{\includegraphics[angle=0,width=.45\linewidth,keepaspectratio,clip,trim=150 30 100
120]{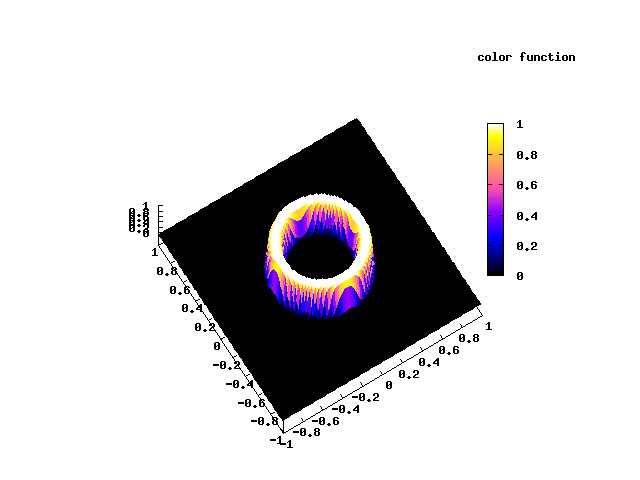}}
\hfill%
}

\caption{Initial data for the multidimensional test.}
\end{figure}

\begin{figure}[!ht]
\centering
{
\hfill%
\subfloat[\label{fig-multiD-wA}Solution w at
$t=0.5$]{\includegraphics[angle=0,width=.38\linewidth,keepaspectratio,clip,trim=150 30 100
120]{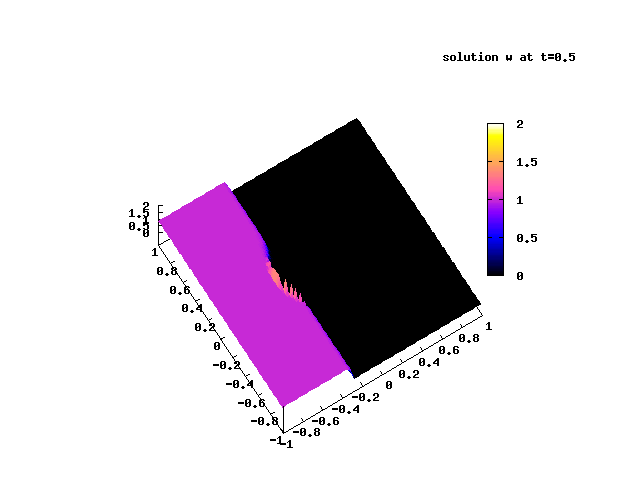}}\hfill%
\subfloat[\label{fig-multiD-uA}Solution u at
$t=0.5$]{\includegraphics[angle=0,width=.38\linewidth,keepaspectratio,clip,trim=150 30 100
120]{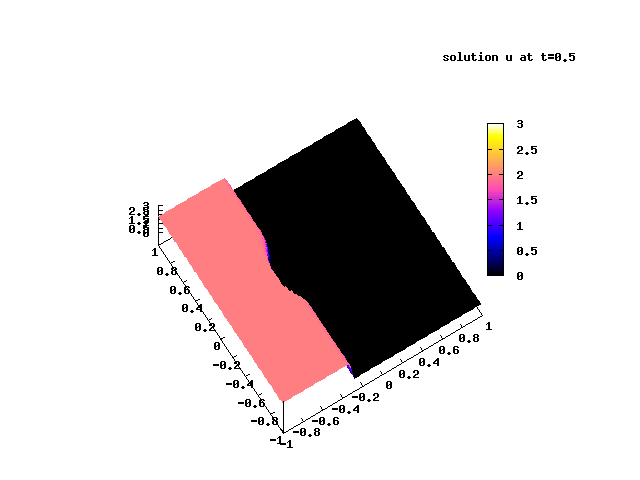}}%
\hfill%
}

{
\hfill%
\subfloat[\label{fig-multiD-wC}Solution w at
$t=1.5$]{\includegraphics[angle=0,width=.38\linewidth,keepaspectratio,clip,trim=150 30 100
120]{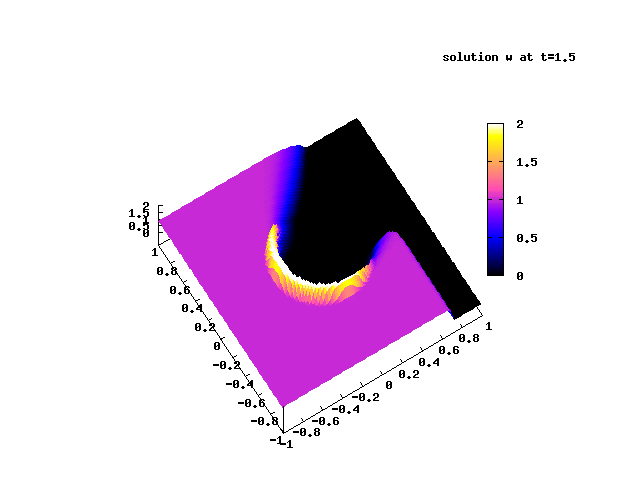}}\hfill%
\subfloat[\label{fig-multiD-uC}Solution u at
$t=1.5$]{\includegraphics[angle=0,width=.38\linewidth,keepaspectratio,clip,trim=150 30 100
120]{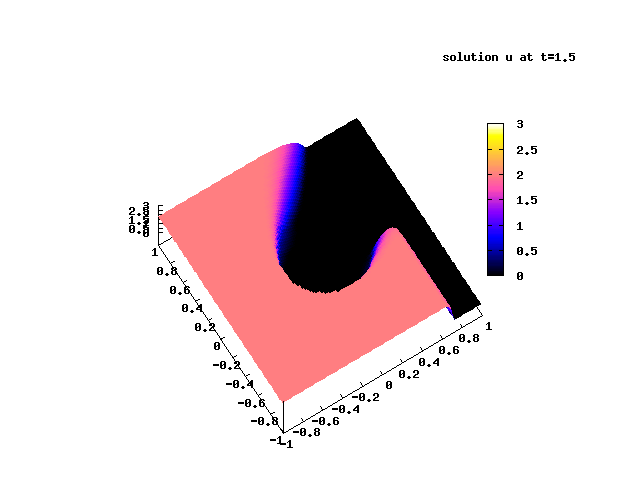}}
\hfill%
}

{
\hfill%
\subfloat[\label{fig-multiD-wE}Solution w at
$t=2.5$]{\includegraphics[angle=0,width=.38\linewidth,keepaspectratio,clip,trim=150 30 100
120]{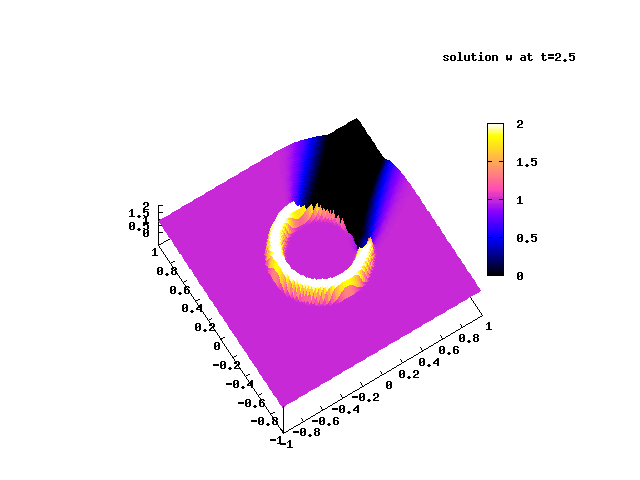}}\hfill%
\subfloat[\label{fig-multiD-uE}Solution u at
$t=2.5$]{\includegraphics[angle=0,width=.38\linewidth,keepaspectratio,clip,trim=150 30 100
120]{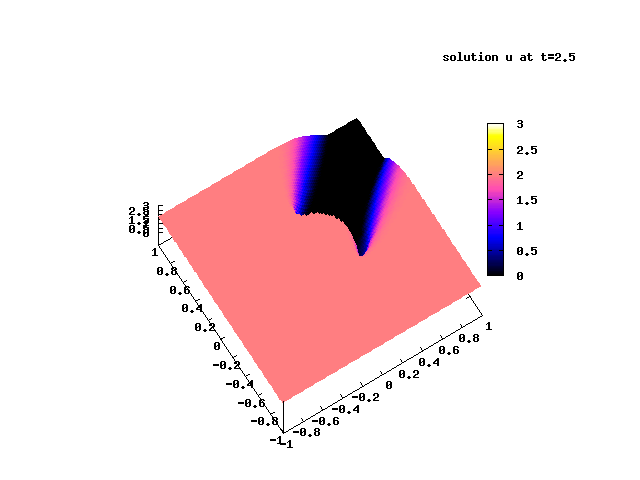}}
\hfill%
}

{
\hfill%
\subfloat[\label{fig-multiD-wH}Solution w at
$t=4.5$]{\includegraphics[angle=0,width=.38\linewidth,keepaspectratio,clip,trim=150 30 100
100]{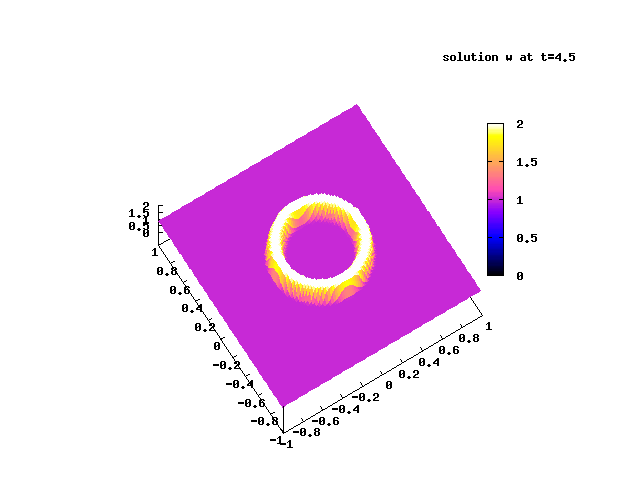}}\hfill%
\subfloat[\label{fig-multiD-uH}Solution u at
$t=4.5$]{\includegraphics[angle=0,width=.38\linewidth,keepaspectratio,clip,trim=150 30 100
100]{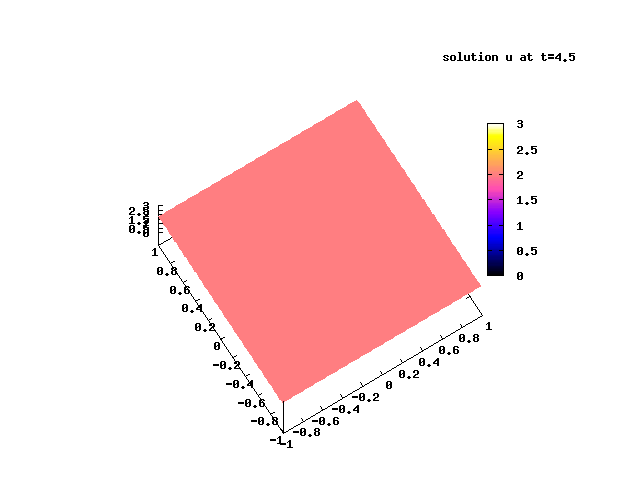}}
\hfill%
}

\caption{Evolution of the solution for different times : $w$ (left) and $u$ (right).}
\end{figure}

\subsection{A three domain coupling problem}

In this second test, we consider three different domains, as represented by the two components of $v$ (see
Figure~\ref{threedomains}). The domain $\calD_2$ is a triangular inclusion and the domain $\calD_1$ is the
complement of $\calD_2$ relative to an annular inclusion. The flux--functions under consideration are now 
\be
  f_0(w)={w^2 \over 2} \begin{pmatrix}1\cr 0\end{pmatrix},\quad
  f_1(w)={w^2 \over 2} \begin{pmatrix}0.5\cr 0\end{pmatrix},\quad
  f_2(w)={w^2\over 2} \begin{pmatrix}0\cr 1\end{pmatrix},
\ee
and the coupling relations are given by the change of unknown \eqref{chgtvarrrr} with 
\be
 \cpl_0(w)=w,\quad \cpl_1(w)=w/2,\quad \cpl_2(w)=w/3.
\ee
We consider the same initial data as previously and, thus, we expect the state $w=2$ to appear in $\calD_1$ and the state
 $w=3$ in $\calD_2$. The results are represented in Figures~\ref{fig-multiD3-t1}
to~\ref{fig-multiD3-t6} for successive time steps. Once again, the limiting solution as the time grows
satisfies the expected coupling relation.

\begin{figure}[!ht]
\subfloat[Domain $\calD_1$.]{\includegraphics[angle=0,width=.45\linewidth,keepaspectratio,clip,trim=200
100 150 110]{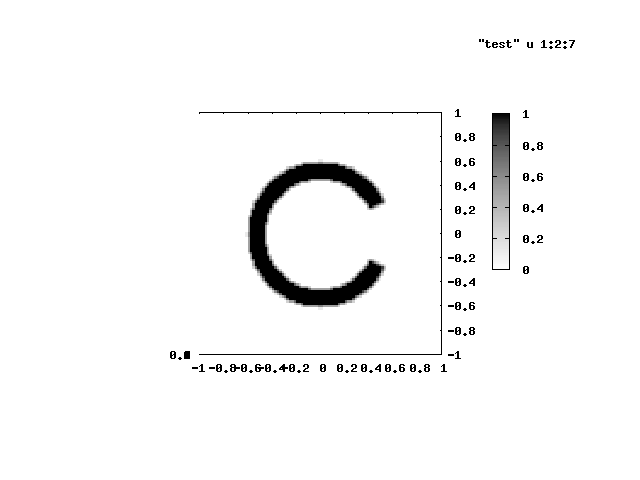}}\hfill%
\subfloat[Domain $\calD_2$.]{\includegraphics[angle=0,width=.45\linewidth,keepaspectratio,clip,trim=200
100 150 110]{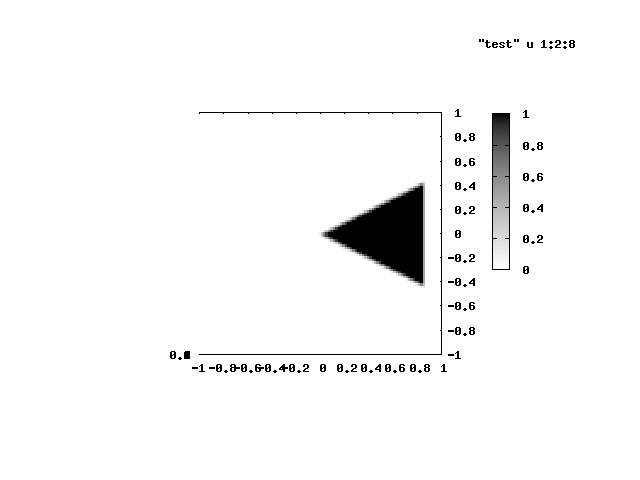}}
 \caption{Geometry of the three domains.}
 \label{threedomains}
\end{figure}

\begin{figure}[!ht]
\centering

{
\hfill%
\subfloat[\label{fig-multiD3-t1}Solution w at
$t=1.0$]{\includegraphics[angle=0,width=.38\linewidth,keepaspectratio,clip,trim=150 30 100
100]{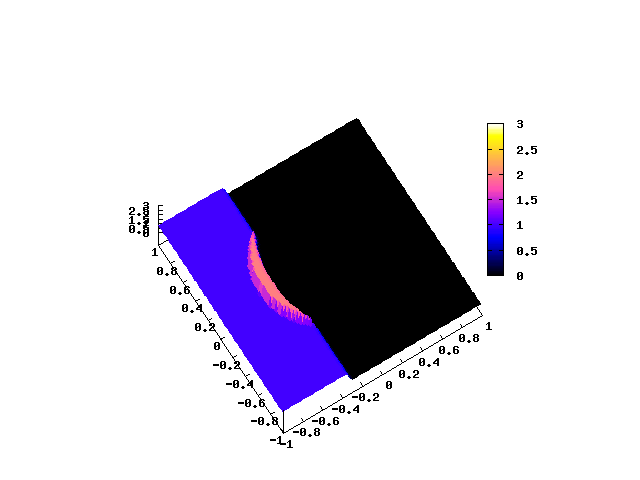}}\hfill%
\subfloat[\label{fig-multiD3-t2}Solution w at
$t=2.0$]{\includegraphics[angle=0,width=.38\linewidth,keepaspectratio,clip,trim=150 30 100
100]{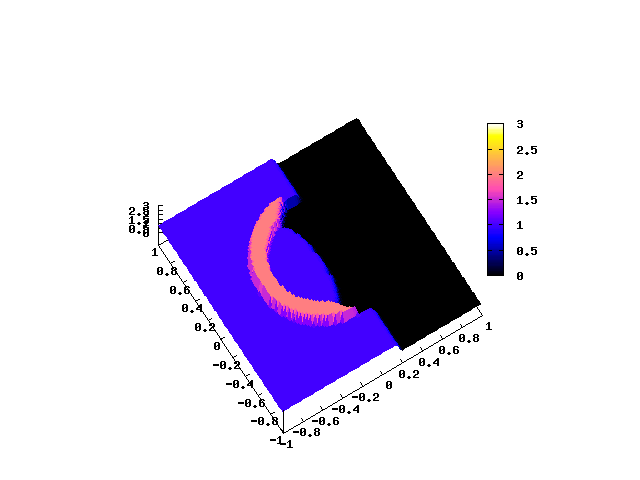}}
\hfill%
}

{
\hfill%
\subfloat[\label{fig-multiD3-t3}Solution w at
$t=3.0$]{\includegraphics[angle=0,width=.38\linewidth,keepaspectratio,clip,trim=150 30 100
100]{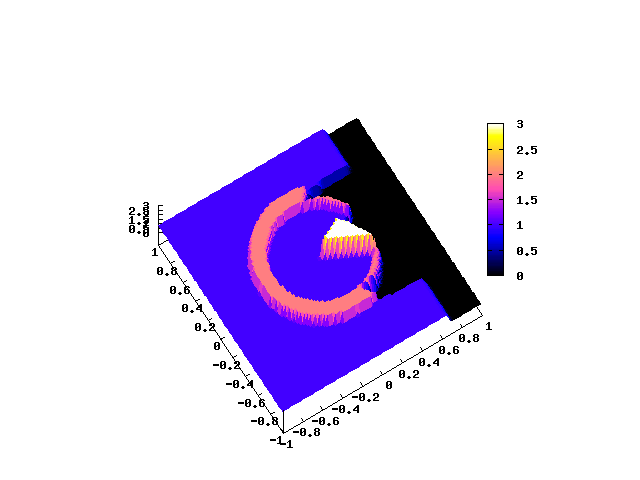}}\hfill%
\subfloat[\label{fig-multiD3-t4}Solution w at
$t=4.0$]{\includegraphics[angle=0,width=.38\linewidth,keepaspectratio,clip,trim=150 30 100
100]{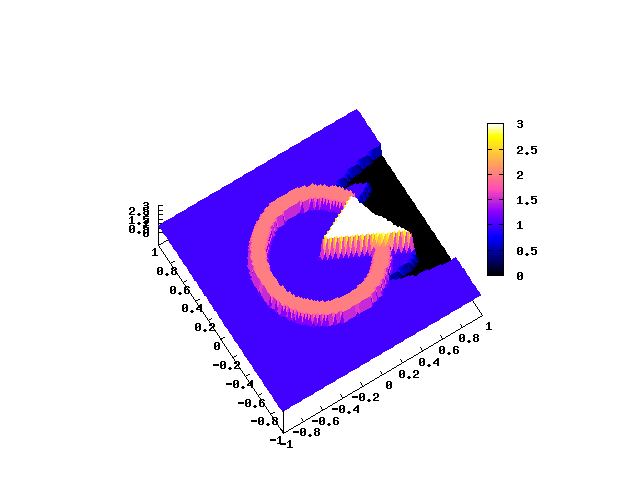}}
\hfill%
}

{
\hfill%
\subfloat[\label{fig-multiD3-t5}Solution w at
$t=5.0$]{\includegraphics[angle=0,width=.38\linewidth,keepaspectratio,clip,trim=150 30 100
100]{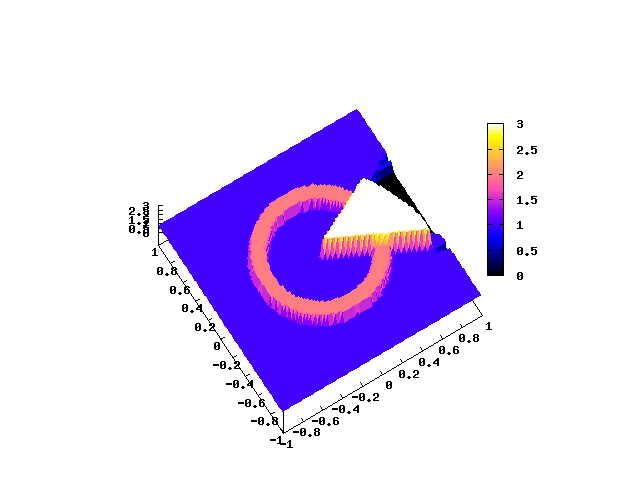}}\hfill%
\subfloat[\label{fig-multiD3-t6}Solution w at
$t=6.0$]{\includegraphics[angle=0,width=.38\linewidth,keepaspectratio,clip,trim=150 30 100
100]{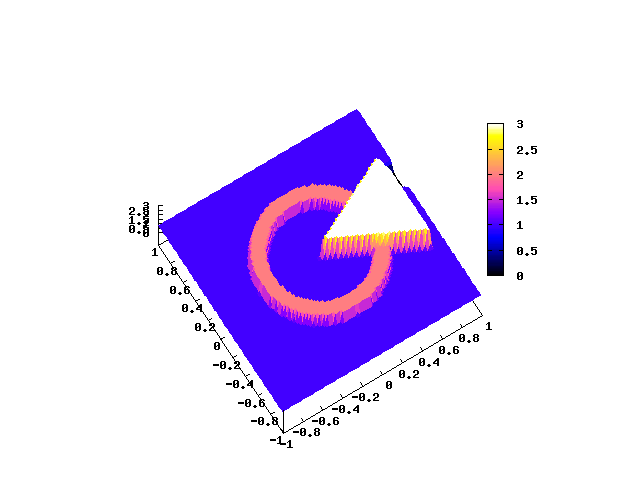}}
\hfill%
}

 \caption{Three domain evolution. Solution $w$.}
\end{figure}


\newcommand{\auth}{\textsc}

\end{document}